\documentclass[12pt]{amsart}
\usepackage[dvips]{graphicx}

\theoremstyle{plain}
	\newtheorem{theorem}{Theorem}[section]
	\newtheorem{lemma}[theorem]{Lemma}
	\newtheorem{corollary}[theorem]{Corollary}
	\newtheorem{definition}[theorem]{Definition}
	\newtheorem{proposition}[theorem]{Proposition}
	\newtheorem{remark}[theorem]{Remark}
	\newtheorem{conjecture}[theorem]{Conjecture}
	\newtheorem{example}[theorem]{Example}

\theoremstyle{plain}

\def\R{\mathbb{R}}
\def\RN{\mathbb{R}^N}
\def\calF{\mathcal{F}}
\def\calL{\mathcal{L}}
\def\e{\varepsilon}
\def\bu{\bar{u}}
\def\tu{\tilde{u}}

\newcommand{\unifcomp}[1]{L^{#1}_{\text{ul,}\rho}(\mathbb{R}^N)}

\makeatletter
\newcommand{\xRightarrow}[2][]{%
\ext@arrow 0055{\Rightarrowfill@}{#1}{#2}%
}
\def\Rightarrowfill@{\arrowfill@\Relbar\Relbar\Rightarrow}
\newcommand{\xLeftarrow}[2][]{%
\ext@arrow 0055{\Leftarrowfill@}{#1}{#2}%
}
\def\Leftarrowfill@{\arrowfill@\Leftarrow\Relbar\Relbar}
\newcommand{\xLongleftrightarrow}[2][]{%
\ext@arrow 0055{\llrafill@}{#1}{#2}%
}
\def\llrafill@{\arrowfill@\Leftarrow\Relbar\Rightarrow}
\makeatother

\makeatletter
 \@addtoreset{equation}{section}
\makeatother

\begin{document}


\title[Fractional semilinear hear equation]{Fractional semilinear heat equations\\
with singular and nondecaying initial data\\
}

\author{Th\'eo Giraudon}
\address{\'Ecole Nationale des Ponts et Chauss\'ees,
6-8 avenue Blaise-Pascal Champs-sur-Marne 77455, Marne-la-Vall\'ee, France}
\email{theo.giraudon@eleves.enpc.fr}

\author{Yasuhito Miyamoto}
\address{Graduate School of Mathematical Sciences, The University of Tokyo,
3-8-1 Komaba Meguro-ku Tokyo 153-8914, Japan}
\email{miyamoto@ms.u-tokyo.ac.jp}

\begin{abstract}
We study integrability conditions for existence and nonexistence of a local-in-time integral solution of fractional semilinear heat equations with rather general growing nonlinearities in uniformly local $L^p$ spaces.
Our main results about this matter consist of Theorems~\ref{A}, \ref{B}, \ref{S5T1} and \ref{S5T2}.
We introduce a new supersolution which plays a crucial role.
Our method does not rely on a change of variables, and hence it can be applied to a wide class of nonlocal parabolic equations.
In particular, when the nonlinear term is $u^p$ or $e^u$, a local-in-time solution can be constructed in the critical case, and integrability conditions for the existence and nonexistence are completely classified.
Our analysis is based on the comparison principle, Jensen's inequality and $L^p$-$L^q$ type estimates.
\end{abstract}
\date{\today}
\maketitle

\section{Introduction and main results}
We are interested in existence and nonexistence of a local-in-time solution of the Cauchy problem
\begin{equation}\label{S1E1}
\begin{cases}
\partial_tu+(-\Delta)^{\theta/2} u=f(u) & \textrm{in}\ \RN\times (0,T),\\
u(x,0)=\phi(x) & \textrm{in}\ \RN,
\end{cases}
\end{equation}
where the domain is $\RN$, $N\ge 1$, the initial function $\phi$ is nonnegative, $\phi$ may be unbounded and nondecaying and $(-\Delta)^{\theta/2}$, $0<\theta\le 2$, denotes the fractional power of the Laplace operator $-\Delta$ on $\RN$.
Throughout the present paper, we define $F(u)$ by
\[
F(u):=\int_u^{\infty}\frac{dt}{f(t)}.
\]
We impose the following assumptions on $f$:
\begin{multline}\label{F1}
f\in C^1(0,\infty)\cap C[0,\infty),\ 
f(u)>0\ \textrm{for}\ u>0,\ 
f'(u)\ge 0\ \textrm{for}\ u\ge 0,\\
\textrm{$F(u)<\infty$ for large $u>0$}\ 
\textrm{and the limit $q:=\lim_{u\to\infty}f'(u)F(u)$ exists.}\tag{F1}
\end{multline}
By (\ref{F1}) we see that $F$ is defined on $(0,\infty)$ and $0<F(0)\le\infty$.
The inverse function $F^{-1}(u)$ exists, because $F(u)$ is strictly decreasing.
When $f\in C^2$, by L'Hospital's role we formally have
\[
q=\lim_{u\to\infty}\frac{F(u)}{1/f'(u)}=\lim_{u\to\infty}\frac{F(u)'}{(1/f'(u))'}=\lim_{u\to\infty}\frac{f'(u)^2}{f(u)f''(u)}.
\]
The growth rate of $f$ can be defined by $p:=\lim_{u\to\infty}uf'(u)/f(u)$.
We formally obtain
\[
p=\lim_{u\to\infty}\frac{(u)'}{(f(u)/f'(u))'}=\lim_{u\to\infty}\frac{1}{1-\frac{f(u)f''(u)}{f'(u)^2}}=\frac{q}{q-1},
\ \ \textrm{and hence}\ \frac{1}{p}+\frac{1}{q}=1.
\]
The exponent $q$, which was introduced in Dupaigne-Farina~\cite{DF10}, is the conjugate of the growth rate of $f$.
The algebraic growth corresponds to $q>1$, while the exponential growth corresponds to $q=1$.
We will show in Section~2 that $q\ge 1$ if $q$ exists.

Let us consider the classical case $\theta=2$, i.e., 
\begin{equation}\label{S1E2}
\begin{cases}
\partial_tu-\Delta u=f(u) & \textrm{in}\ \RN\times (0,T),\\
u(x,0)=\phi(x) & \textrm{in}\ \RN.
\end{cases}
\end{equation}
Weissler~\cite{W80} started studying the solvability of (\ref{S1E2}) with possibly unbounded and sign-changing initial data $\phi\in L^r(\RN)$ and found the critical exponent $N(p-1)/2$ as described in Proposition~\ref{S1P1}.
In the model case $f(u)=|u|^{p-1}u$, $p>1$, the solvability in \cite{W80} can be summarized as follows:
\begin{proposition}\label{S1P1}
Let $N\ge 1$.
Assume that $f(u)=|u|^{p-1}u$, $p>1$.
The following hold:\\
(i)(Existence, subcritical case) Assume $r\ge 1$ and $r>{N(p-1)}/2$. The problem (\ref{S1E2}) has a local-in-time solution for $\phi\in L^r(\RN)$.\\
(ii)(Existence, critical case) Assume $r={N(p-1)}/2>1$. The problem (\ref{S1E2}) has a local-in-time solution for $\phi\in L^r(\RN)$.\\
(iii)(Nonexistence, supercritical case) For each $1\le r<N(p-1)/2$, there is $\phi\in L^r(\RN)$ such that (\ref{S1E2}) has no local-in-time nonnegative solution.
\end{proposition}
Let $u_{\lambda}(x,t):=\lambda^{2/(p-1)}u(\lambda x,\lambda^2t)$.
When $u(x,t)$ satisfies the equation in (\ref{S1E2}) with $f(u)=|u|^{p-1}u$, the function $u_{\lambda}(x,t)$ also satisfies the same equation.
We see that $$\left\|u_{\lambda}(\,\cdot\,,0)\right\|_{L^r(\RN)}=\left\|u(\,\cdot\,,0)\right\|_{L^r(\RN)},\quad\lambda>0$$
 if and only if $r=N(p-1)/2$.
Proposition~\ref{S1P1} indicates that $L^{N(p-1)/2}(\RN)$ becomes a borderline space for the solvability of the equation.
The problem (\ref{S1E2}) has been studied by many authors.
See \cite{AD91,BC96,G86,LRSV16,RS13,W80,W86} for example.
The reader can consult Quittner-Souplet~\cite{QS07} and references therein.

Hereafter, let $L^r(\Omega)$, $1\le r\le\infty$, denote the usual Lebesgue space on the domain $\Omega$.
We write $\left\|u\right\|_r=\left\|u\right\|_{L^r(\Omega)}$ for simplicity when $\Omega=\RN$.
In order to deal with singular and nondecaying functions we define uniformly local $L^r$ spaces:
\[
L^r_{\rm{ul},\rho}(\RN):=\left\{ u\in L^1_{\rm{loc}}(\RN)\left|\ \left\|u\right\|_{L^r_{\rm{ul},\rho}(\RN)}<\infty\right.\right\}.
\]
Here, $\rho>0$, $B_y(\rho):=\{x\in\RN|\ |x-y|<\rho\}$ and
\[
\left\|u\right\|_{L^r_{\rm{ul},\rho}(\RN)}:=
\begin{cases}
\sup_{y\in\RN}\left(\int_{B_y(\rho)}|u(x)|^rdx\right)^{1/r} & \textrm{if}\ \ 1\le r<\infty,\\
{\rm{esssup}}_{y\in\RN}\left\| u\right\|_{L^{\infty}(B_y(\rho))} & \textrm{if}\ \ r=\infty.
\end{cases}
\]
We easily see that $L^{\infty}_{\rm{ul},\rho}(\RN)=L^{\infty}(\RN)$ and that $L^{r_1}_{\rm{ul},\rho}(\RN)\subset L^{r_2}_{\rm{ul},\rho}(\RN)$ if $1\le r_2\le r_1\le\infty$.
We define $\calL^r_{\rm{ul},\rho}(\RN)$ by
\[
\calL^r_{\rm{ul},\rho}(\RN):=
\overline{BUC(\RN)}^{\|\,\cdot\,\|_{L^r_{\rm{ul},\rho}(\RN)}},
\]
i.e., $\calL^r_{\rm{ul},\rho}(\RN)$ denotes the closure of the space of bounded uniformly continuous functions $BUC(\RN)$ in the space $L^r_{\rm{ul},\rho}(\RN)$.
By working in $L^r_{\rm{ul},\rho}(\RN)$ (or $\calL^r_{\rm{ul},\rho}(\RN)$) instead of $L^r(\RN)$ we can focus on the relationship between the singularity of $\phi$ and the solvability, and eliminate the effect of the behavior of $\phi$ near space infinity.

Fujishima-Ioku~\cite{FI18} studied the solvability of (\ref{S1E2}) in $L^r_{\rm{ul},\rho}(\RN)$ with nonnegative initial data $\phi$.
Specifically, they obtained the following:
\begin{proposition}\label{S1P2}
Let $N\ge 1$ and $\theta=2$.
Assume that $f$ satisfies (\ref{F1}) with $q\ge 1$ and that $\phi\ge 0$.
Then the following hold:\\
(i)(Existence, subcritical case) Assume that $r>N/2$ and $r\ge q-1$ and that $f'(u)F(u)\le q$ for large $u>0$. If $F(\phi)^{-r}\in L^1_{\rm{ul},\rho}(\RN)$, then (\ref{S1E1}) has a local-in-time solution.\\
(ii)(Existence, critical case) Assume that $r=N/2>q-1$ and that $f'(u)F(u)\le q$ for large $u>0$. If $F(\phi)^{-r}\in \calL^1_{\rm{ul},\rho}(\RN)$, then (\ref{S1E1}) has a local-in-time solution.\\
(iii)(Nonexistence, supercritical case) Assume that $f$ is convex and $q-1<N/2$. For any $r\in[q-1,N/2)$ if $q>1$ or $r\in (0,N/2)$ if $q=1$, there is a nonnegative $\phi$ such that $F(\phi)^{-r}\in L^1_{\rm{ul},\rho}(\RN)$ and (\ref{S1E1}) has no local-in-time nonnegative solution.
\end{proposition}
Note that $F(u)^{-r}=1/(F(u)^r)$, while $F^{-1}(u)$ stands for the inverse function of $F(u)$.
Because of the existence of the $q$ exponent in (\ref{F1}), the equation (\ref{S1E1}) with $\theta=2$ is approximately invariant under the quasi-scaling
\[
u_{\lambda}(x,t):=F^{-1}\left(\lambda^{-2}F(u(\lambda x,\lambda^2t))\right),\ \ \lambda>0
\]
and the invariance
\[
\int_{\RN}F(u_{\lambda}(x,0))^{-N/2}dx=\int_{\RN}F(u(x,0))^{-N/2}dx,\ \ \lambda>0
\]
gives the borderline set in Proposition~\ref{S1P2}, which may not be a space.
If $f(u)=u^p$, then $F(u)^{-N/2}=(p-1)^{N/2}u^{N(p-1)/2}$, and hence Proposition~\ref{S1P2} is a generalization of Proposition~\ref{S1P1} for nonnegative initial data.
The leading term is not necessarily a pure power, and it can grow more rapidly than the single exponential function.
Two nonlinearities $f(u)=u^p(\log(u+1))$ and $e^{u^p}$, $p\ge 1$,  are included as examples.
See Example~2.2 of the present paper.
In \cite{FI18} the authors introduced interesting changes of variables \cite[Eqs (1.18), (1.19)]{FI18} which is denoted by $\tilde{u}=T(u)$. 
Using these changes of variables, one can transform (\ref{S1E2}) not exactly but approximately into one of the canonical two equations:
\begin{equation}\label{fq}
\partial_t\tu-\Delta \tu=f_q(\tu),\ \ \textrm{where}\ \ f_q(u):=
\begin{cases}
u^p,\ \ \frac{1}{p}+\frac{1}{q}=1, & \textrm{if}\ q>1,\\
e^u & \textrm{if}\ q=1.
\end{cases}
\end{equation}
See (\ref{FI}) for the exact equation for $\tu$.
This method does not work in the case (\ref{S1E1}), because it is difficult to obtain a relationship between $(-\Delta)^{\theta/2}u$ and $(-\Delta)^{\theta/2}\tilde{u}$.
Therefore, a generalization of Proposition~\ref{S1P2} to nonlocal equations is not obvious.

Let us go back to fractional equations.
We focus on the local-in-time solution with unbounded initial data.
By \cite[Theorem~3]{W80} we easily conclude an existence theorem corresponding to Proposition~\ref{S1P1}~(i) and (ii) when $0<\theta\le 2$.
Li~\cite{L17} considered (\ref{S1E1}) with $1<\theta<2$ when $f$ is continuous and nondecreasing and $\phi\in L^r(\RN)$ is nonnegative.
He showed, among other things, that, for $1<\theta<2$ and $1<r<\infty$, (\ref{S1E1}) has a local-in-time solution if and only if $\limsup_{s\to 0}f(s)/s<\infty$ and $\lim_{s\to\infty}s^{-(1+\theta r/N)}f(s)<\infty$.
Therefore, for $1<r<\infty$, the problem (\ref{S1E1}) with $f(u)=u^p$, $p>1$, has a local-in-time solution if and only if $r\ge N(p-1)/\theta$.
His method was based on Laister {\it et~al.}$\,$\cite{LRSV16} which obtained necessary and sufficient conditions on $f$ for the existence of a local-in-time nonnegative solution for (\ref{S1E1}) with $\theta=2$.
For $1<\theta\le 2$ and $1\le r<\infty$,  Li~\cite{L17b} constructed a nonnegative nondecreasing nonlinear term $f$ and a nonnegative initial data $\phi\in L^r(\RN)$ such that $\int_1^{\infty}ds/f(s)=\infty$ and (\ref{S1E1}) has no local-in-time solution in $L^1_{\rm{loc}}(\RN)$.
Hisa-Ishige~\cite{HI18} studied (\ref{S1E1}) with $f(u)=u^p$ when initial data $\phi$ is a Radon measure.
They obtained necessary conditions and sufficient conditions for a local-in-time existence.
Moreover, the authors obtained an optimal singularity which depends on $p>1$.
For example, they showed that for $\phi(x)=\gamma|x|^{-\theta/(p-1)}$ and $p>1+\theta/N$, there is $\gamma^*>0$ such that (\ref{S1E1}) with $f(u)=u^p$ has (resp. does not have) a local-in-time solution if $0\le\gamma<\gamma^*$ (resp. $\gamma>\gamma^*$).
Let $L^{N(p-1)/\theta,\infty}(\RN)$ denote a Lorentz space, which we do not define in the present paper.
It is known that
\[
\begin{tabular}{ccccc}
$L^{{N(p-1)}/{\theta}}(\RN)$ & $\subset$ & $L^{{N(p-1)}/{\theta},\infty}(\RN)$ & $\subset$ & $L_{\rm{loc}}^{{N(p-1)}/{\theta}-\e}(\RN)$ $(\e>0\ \textrm{is small.})$, \\
\rotatebox[origin=c]{90}{$\not\in$} & & \rotatebox[origin=c]{90}{$\in$} && {}\\
$|x|^{-\frac{\theta}{p-1}}$ && $|x|^{-\frac{\theta}{p-1}}$ &&  
\end{tabular}
\]
which describes the borderline property of the singularity of $|x|^{-\theta/(p-1)}$ in $L^{N(p-1)/\theta}(\RN)$.
See Furioli {\it et~al.}$\,$\cite{FKRT17} for the case $f(u)=|u|^{\frac{r\theta}{N}}ue^{u^r}$ in the Orlicz space $\exp L^r(\RN)$.
Note that $\exp L^r(\RN)\subset L^q(\RN)$ for every $q\in[r,\infty)$.

Let us define a solution of (\ref{S1E1}) in $L^r_{\rm{ul},\rho}(\RN)$.
The linear fractional heat equation
\begin{equation}\label{S1E4}
\partial_tu+(-\Delta)^{\theta/2}u=0
\end{equation}
admits a fundamental solution $G$.
We recall various properties of $G$ in Section~2.
Because of the decay estimate (\ref{Gdecay}), the following definition of a solution of (\ref{S1E4}) with initial data $\phi$ becomes well-defined: For $\phi\in L^1_{\rm{ul},\rho}(\RN)$,
\[
S(t)[\phi](x):=\int_{\RN}G(x-y,t)\phi(y)dy,\ \ x\in\RN,\ t>0.
\]
For the optimal class of the positive initial data for $S(t)$, see Bonforte {\it et~al.}$\,$\cite{BSV17}.
\begin{definition}[Integral solution]\label{S1D1}
Let $u$ and $\bu$ be nonnegative measurable functions on $\RN\times (0,T)$.
We call $u$ a solution of (\ref{S1E1}) if $u$ satisfies the integral equation
\[
\infty>u(t)=\calF[u](t)\ \ \textrm{a.e.}\ x\in\RN,\ \ 0<t<T,
\]
where
\[
\calF[u](t):=S(t)\phi+\int_0^tS(t-s)f(u(s))ds.
\]
We call $\bu$ a supersolution of (\ref{S1E1}) if $\bu$ satisfies the integral inequality
\[
\calF[\bu](t)\le\bu(t)\ \ \textrm{a.e.}\ x\in\RN,\ \ 0<t<T.
\]
\end{definition}
Note that a mild solution is an integral solution defined by Definition~\ref{S1D1}.
See \cite[p. 78]{QS07}.
In this paper we do not pursue in which function space the solution $u(t)$ converges to the initial data $\phi$ as $t\to 0$.

The first main result is about a power case.
\begin{theorem}[{\bf Algebraic growth {\boldmath $q>1$}}]\label{A}
Let $N\ge 1$ and $0<\theta\le 2$.
Assume that $f$ satisfies (\ref{F1}) with $q>1$ and that $\phi\ge 0$.
Then the following hold:\\
(i-1)(Existence, subcritical case 1) Assume that $r>N/\theta$ and $r>q-1$.
If $F(\phi)^{-r}\in L^1_{\rm{ul},\rho}(\RN)$, then (\ref{S1E1}) has a local-in-time solution in the sense of Definition~\ref{S1D1}.\\
(i-2)(Existence, subcritical case 2) Assume that $r>N/\theta$ and $r\ge q-1$ and that $f'(u)F(u)\le q$ for large $u>0$. If $F(\phi)^{-r}\in L^1_{\rm{ul},\rho}(\RN)$, then (\ref{S1E1}) has a local-in-time solution in the sense of Definition~\ref{S1D1}.\\
(ii)(Nonexistence, supercritical case) Assume that $f$ is convex. For any $r\in (0,N/\theta)$, there is $\phi\ge 0$ such that $F(\phi)^{-r}\in L^1_{\rm{ul},\rho}(\RN)$ and (\ref{S1E1}) has no local-in-time nonnegative solution in the sense of Definition~\ref{S1D1}.\\
\end{theorem}

\begin{remark}\label{R1}
(i) If $f(u)=u^p$, then $F(u)^{-N/\theta}=(p-1)^{N/\theta}u^{N(p-1)/\theta}$, and hence the critical exponent becomes $N(p-1)/\theta$, which was presented in \cite{HI18,L17b,W80}.
Hence, Theorem~\ref{A} is a generalization of Proposition~\ref{S1P2}~(i) and (iii) when $q>1$.\\
(ii) Even in the case $\theta=2$ Theorem~\ref{A} includes a new result.
The inequality $f'(u)F(u)\le q$ for large $u>0$ is assumed in \cite[Theorem~1.1]{FI18}, while Theorem~\ref{A}~(i-1) does not assume.
When $f(u)=(u+e)^p/(\log (u+e))$ ($1/p+1/q=1$), we see that $\lim_{u\to\infty}f'(u)F(u)=q$ and that $f'(u)F(u)>q$ for large $u>0$.
Theorem~\ref{A} (i-1) is applicable.
\\
(iii) The critical case $r=N/\theta>q-1$, which corresponds to Proposition~\ref{S1P2}~(ii), is not included in Theorem~\ref{A}.
However, if $f(u)=u^p$ or $e^u$, then in Section~5 we construct a local-in-time solution in $\calL^r_{\rm{ul},\rho}(\RN)$.
In particular, Theorem~\ref{S5T1} is an $L^r_{\rm{ul},\rho}(\RN)$ version of Proposition~\ref{S1P1} in the case $\phi\ge 0$.
When $f(u)=u^p$ and $\phi\ge 0$, Theorem~\ref{S5T1} gives a complete classification in $L^r_{\rm{ul},\rho}(\RN)$.\\
(iv) In Theorem~\ref{A}~(ii) we take $\phi\ge 0$ such that $\phi(x)=F^{-1}(|x|^{\alpha})$, $\theta<\alpha<N/r$, near the origin.
We can choose $\phi$ such that $\phi\in L^1_{\rm{ul},\rho}(\RN)$ if $r\ge q-1$.
\end{remark}

We consider the exponential growth case $q=1$.
In the next main theorem we assume the following:
\begin{equation}\label{F2}
\textrm{$f(u)$ is convex for large $u>0$, and $f'(u)F(u)\le 1$ for large $u>0$.}\tag{F2}
\end{equation}
\begin{theorem}[{\bf Exponential growth {\boldmath $q=1$}}]\label{B}
Let $N\ge 1$ and $0<\theta\le 2$.
Assume that $f$ satisfies (\ref{F1}) with $q=1$ and that $\phi\ge 0$.
Then the following hold:\\
(i)(Existence, subcritical case) Assume that $r>N/\theta$ and that (\ref{F2}) holds.
If $F(\phi)^{-r}\in L^1_{\rm{ul},\rho}(\RN)$, then (\ref{S1E1}) has a local-in-time solution in the sense of Definition~\ref{S1D1}.\\
(ii)(Nonexistence, supercritical case) Assume that $f(u)$ is convex for $u\ge 0$. For any $r\in (0,N/\theta)$, there is $\phi\ge 0$ such that $F(\phi)^{-r}\in L^1_{\rm{ul},\rho}(\RN)$ and (\ref{S1E1}) does not have a local-in-time nonnegative solution in the sense of Definition~\ref{S1D1}.
\end{theorem}
\begin{remark}
(i) The nonlinear terms $\exp(u^p)$, $p\ge 1$, and $\exp(\cdots\exp(u)\cdots)$ satisfy both (\ref{F1}) and (\ref{F2}).
The two functions are convex.
Therefore, we can apply Theorem~\ref{B}~(i) and (ii).
See Example~\ref{EX1}.\\
(ii) The critical case is not included in Theorem~\ref{B}.
However, in Theorem~\ref{S5T2} we construct a local-in-time solution for $f(u)=e^u$ in the critical case, as mentioned in Remark~\ref{R1}~(iii).
When $f(u)=e^u$ and $\phi\ge 0$, Theorem~\ref{S5T2} gives a complete classification in $L^r_{\rm{ul},\rho}(\RN)$.\\
(iii) In Theorem~\ref{B}~(ii) we take the same $\phi$ as in Theorem~\ref{A}~(ii). We can choose $\phi$ such that $\phi\in L^1_{\rm{ul},\rho}(\RN)$.
\end{remark}

Let us explain technical details.
The fixed point argument does not manage to construct a solution in $L^r(\RN)$ or $L^r_{\rm{ul},\rho}(\RN)$ when the nonlinear term grows exponentially.
Then we use a monotonicity method (Lemma~\ref{S2L-1}).
See \cite{RS13} for details of the method.
The initial data has to be nonnegative, while we can deal with rapidly growing nonlinearities.
In this method the existence of a supersolution is crucial.
A feature of the present paper is supersolutions (\ref{SS}) and (\ref{SS3}).
They are inspired by \cite{FI18,RS13}.
We do not use changes of variables, and hence those supersolutions enable us to construct a solution of equations with nonlocal operators.
Those supersolutions look natural in view of \cite[Eq (13)]{RS13} and \cite[Eqs (1.18), (1.19)]{FI18}.
In particular, when $f(u)=u^p$ or $e^u$, the supersolutions (\ref{S5T1E0}), (\ref{S5T1E2+-}) and (\ref{S5T2E0}) are surprisingly simple.
However, various estimates in Section~3 are nontrivial, and all estimates in the critical case, which are discussed in Section~5, are optimal in the sense where the exponent of the time variable $t$ becomes $0$.
When $\theta=2$, the way of the construction of the supersolutions can be explained as follows:
Assume that $\theta=2$ and that $f$ satisfies (\ref{F1}) with $q\ge 1$.
Let $u$ be a solution of (\ref{S1E1}) with initial data $\phi$, and let $\tu:=(F_q^{-1}\circ F)(u)$.
Then, $\tu$ satisfies
\begin{equation}\label{FI}
\partial_t\tu-\Delta\tu=f_q(\tu)+\frac{q-f'(u)F(u)}{f_q(\tu)F_q(\tu)}|\nabla\tu|^2.
\end{equation}
The solvability depends on the behavior of $u$ where $u$ is close to $\infty$.
When $u$ is large, by (\ref{F1}) we see that $f'(u)F(u)\approx q$, and hence $\partial_t\tu-\Delta\tu\approx f_q(\tu)$ if $|\nabla\tu|^2$ is not large.
In other words, the general equation (\ref{S1E1}) can be transformed into one of the canonical two equations (\ref{fq}) if $u$ is large.
On the other hand, the solution of the canonical two equations is approximated by a solution of the {\it linear} heat equation in a short time.
Therefore, $\tu(t)\approx (S(t)\circ F_q^{-1}\circ F)(\phi)$ for small $t>0$.
Then, the solution $u$ of the original equation can be approximated by the pullback of the approximate solution of $\partial_t\tu-\Delta\tu\approx f_q(\tu)$, i.e., $u(t)\approx (F^{-1}\circ F_q\circ S(t)\circ F_q^{-1}\circ F)(\phi)$.
See Figure~\ref{fig1}.
Modifying this approximate solution, we obtain the supersolutions (\ref{SS}) and (\ref{SS3}).
These supersolutions work well even in the case $0<\theta<2$.
\begin{figure}\label{fig1}
\begin{tabular}{ccc}
$\partial_t u-\Delta u=f(u)$ & 
$\xLongleftrightarrow[{\rm{Almost~equivalent~if~{\it u}~is~large.}}]{{F(u)=F_q(\tu)}}$ &
$\partial_t\tu-\Delta\tu\approx f_q(\tu)$
\\
\\
$\phi$ & $\longmapsto$ & $(F_q^{-1}\circ F)(\phi)$ \\
 & &$\qquad\qquad\qquad$ \rotatebox[origin=c]{270}{$\longmapsto$} $t>0$ is small. \\
$(F^{-1}\circ F_q\circ S(t)\circ F_q^{-1}\circ F)(\phi)$ & \rotatebox[origin=c]{180}{$\longmapsto$} & $(S(t)\circ F_q^{-1}\circ F)(\phi)$
\end{tabular}\caption{The approximation of the solution via the canonical two equations.}
\end{figure}
The proof of the nonexistence is rather standard.
Specifically, the decay estimate (\ref{S4PBE0+}) of the solution with the initial data (\ref{S4PBE0}) contradicts the necessary condition for the local-in-time existence given in Proposition~\ref{S4P1}.
However, the exponential decay of the heat kernel is not used in the proof.
Sharp estimates, which are newly obtained in this paper, are required. See Lemmas~\ref{S4L3} and \ref{S4L3+}.

Let us mention $L^p_{\rm{ul},\rho}(\RN)$ spaces.
The $L^p_{\rm{ul},\rho}$-$L^q_{\rm{ul},\rho}$ estimate for convolution type operators was obtained by Maekawa-Terasawa~\cite[Theorem~3.1]{MT06}.
The proof works for the case $0<\theta<2$ with modifications.

This paper consists of six sections.
In Section~2 we give examples of Theorems~\ref{A} and \ref{B}.
We recall properties of the fundamental solution $G$.
We prove basic results which will be used in the proof of main theorems.
In Section~3 we prove existence theorems, i.e., Theorems~\ref{A} (i-1) (i-2) and \ref{B} (i).
In Section~4 we prove nonexistence theorems, i.e., Theorems~\ref{A} (ii) and \ref{B} (ii).
In Section~5 we construct a local-in-time solution for $f(u)=u^p$ or $e^u$ in the critical case.
Section~6 is a summary and conjectures.

\section{Examples and preliminaries}
\subsection{Example}
The following Lemma~\ref{Lq} is a fundamental property about the exponent $q$.
\begin{lemma}\label{Lq}
Let $f$ be a function such that (\ref{F1}) with $q$ holds.
Then $q\ge 1$.
\end{lemma}
\begin{proof}
This is proved by \cite[Remark 1.1]{FI18} and \cite[Lemma~2.1]{M18}.
However, we show the proof.
Suppose the contrary, i.e., there is $q_0\in(0,1)$ such that $f'(u)F(u)\le q_0$ for $u\ge u_0$.
Integrating $f'(u)/f(u)\le q_0/(f(u)F(u))$ over $[u_0,u]$ twice, we have
\[
\frac{u-u_0}{f(u_0)F(u_0)^{q_0}}\le\frac{1}{1-q_0}\left(F(u_0)^{1-q_0}-F(u)^{1-q_0}\right)
\ \ \textrm{for}\ \ u>u_0.
\] 
Then,
\[
0\le F(u)^{1-q_0}\le F(u_0)^{1-q_0}-\frac{(1-q_0)(u-u_0)}{f(u_0)F(u_0)^{q_0}}
\to -\infty\ \ \textrm{as}\ \ u\to \infty.
\]
We obtain a contradiction, and hence $q\ge 1$.
\end{proof}
\begin{example}\label{EX1}
(i) Let $f(u):=(u+1)^p\log (u+1)$, $p>1$.
We have
\[
\lim_{u\to\infty}f'(u)F(u)=\lim_{u\to\infty}\frac{f'(u)^2}{f(u)f''(u)}=\frac{p}{p-1}.
\]
Moreover, by direct calculation we see that $f'(u)^2/(f(u)f''(u))\le p/(p-1)$.
Integrating $1/f(u)\le pf''(u)/((p-1)f'(u)^2)$ over $[u,\infty)$, we see that $f'(u)F(u)\le p/(p-1)$ for large $u>0$.
Hence, $f$ satisfies (\ref{F1}) with $q=p/(p-1)$ and $f'(u)F(u)\le q$ for large $u>0$.
Theorem~\ref{A} is applicable.
The leading term of $f$ is not necessarily $u^p$.\\
(ii) Let $f(u)=\exp(u^p)$, $p\ge 1$.
By a similar argument as in (i) we can show that $\lim_{u\to\infty}f'(u)F(u)=1$.
We can easily see that (\ref{F2}) holds.
Therefore, Theorem~\ref{B} is applicable.\\
(iii) Let $f(u)=\exp(\cdots\exp(u)\cdots)$ be the $n$-th iterated exponential function.
Then, $f$ satisfies (\ref{F1}) with $q=1$ and (\ref{F2}).
Theorem~\ref{B} is applicable. See \cite[Section 2.3]{M18}.\\
(iv) Let $f(u):=\exp(g(u))$.
If $g\in C^2(0,\infty)$, $g'(u)\ge 0$ for $u\ge 0$, $g''(u)\ge 0$ for $u\ge 0$ and $\lim_{u\to\infty}g''(u)/(g'(u)^2)=0$, then $f$ satisfies (\ref{F1}) with $q=1$ and (\ref{F2}).
Theorem~\ref{B} is applicable.
\end{example}

\subsection{Fundamental solution}
Let $N\ge 1$ and $0<\theta\le 2$.
The fractional heat equation (\ref{S1E4}) admits a fundamental solution $G$.
We recall various properties of $G$.
The fundamental solution $G(x,t)$ is expressed by
\[
G(x,t)=
\begin{cases}
(4\pi t)^{-N/2}\exp\left(-\frac{|x|^2}{4t}\right) & \textrm{if}\ \theta=2,\\
\int_0^tg_{t,\frac{\theta}{2}}(s)(4\pi s)^{-\frac{N}{2}}\exp\left(-\frac{|x|^2}{4s}\right)ds & \textrm{if}\ 0<\theta<2,
\end{cases}
\]
where $g_{t,\frac{\theta}{2}}$ is a nonnegative function on $[0,\infty)$ defined by
\[
g_{t,\frac{\theta}{2}}(s):=\frac{1}{2\pi i}\int_{\sigma-i\infty}^{\sigma+i\infty}
\exp\left(zs-tz^{\frac{\theta}{2}}\right)dz,\ \sigma>0,\ t>0.
\]
The fundamental solution $G$ is a positive smooth function on $\RN\times (0,\infty)$.
Moreover, the following hold:
\begin{align}
&G(x,t)=t^{-\frac{N}{\theta}}G(t^{-\frac{1}{\theta}}x,1),\nonumber\\
&C^{-1}(1+|x|)^{-N-\theta}\le G(x,1)\le C(1+|x|)^{-N-\theta}\qquad\ \ \textrm{if}\ 0<\theta<2,\label{Gdecay}\\
& G(x,1)=(4\pi)^{-N/2}\exp(-|x|^2/4)\le C(1+|x|)^{-N-\theta}\ \ \textrm{if}\ \ \theta=2,\nonumber\\
&G(\,\cdot\,,1)\ \textrm{is radially symmetric and $G(x,1)\le G(y,1)$ if $|x|\ge|y|$,}\nonumber\\
&G(x,t)=\int_{\RN}G(x-y,t-s)G(y,s)dy,\nonumber\\
&\int_{\RN}G(x,t)dx=1,\nonumber
\end{align}
for $x,y\in\RN$ and $0<s<t$.
See e.g., \cite[Section~2]{S75} for the representation of $G$ and above properties.
These properties are summarized as follows:
\begin{proposition}\label{S2S2P1}
Let $N\ge 1$ and $0<\theta\le 2$.
Then there is a function of one variable $K(\,\cdot\,)$ such that the following hold:
$K(\,\cdot\,)$ is positive nonincreasing, $G(x,t)=t^{-N/\theta}K(t^{-1/\theta}|x|)$ and
there is $C>0$ such that $0\le |x|^{N+\theta}K(|x|)\le C$ for $x\in\RN$.
In particular, $K(|x|)=G(x,1)$.
\end{proposition}

\begin{proposition}\label{S2L3}
Let $N\ge 1$ and $0<\theta\le 2$, and let $1\le\alpha\le\beta\le\infty$.
Then there is $C>0$ such that, for $\phi\in L^{\alpha}_{\rm{ul},\rho}(\RN)$,
\begin{equation}\label{S2L3E0}
\left\|S(t)\phi\right\|_{L^{\beta}_{\rm{ul},\rho}(\RN)}
\le \left({C}{\rho^{-N\left(\frac{1}{\alpha}-\frac{1}{\beta}\right)}}
+{C}{t^{-\frac{N}{\theta}\left(\frac{1}{\alpha}-\frac{1}{\beta}\right)}}\right)
\left\|\phi\right\|_{L^{\alpha}_{\rm{ul},\rho}(\RN)}
\ \ \textrm{for}\ \ t>0.
\end{equation}
Hence, there are $C_0>0$ and $t_0>0$ such that
\begin{equation}\label{S2L3E1}
\left\|S(t)\phi\right\|_{L^{\beta}_{\rm{ul},\rho}(\RN)}
\le C_0t^{-\frac{N}{\theta}\left(\frac{1}{\alpha}-\frac{1}{\beta}\right)}
\left\|\phi\right\|_{L^{\alpha}_{\rm{ul},\rho}(\RN)}
\ \ \textrm{for}\ \ 0<t<t_0.
\end{equation}
\end{proposition}
\begin{proof}
The inequality (\ref{S2L3E0}) follows from the $L^{\alpha}$-$L^{\beta}$ inequality
\begin{equation}\label{S2L3E2}
\left\|S(t)\phi\right\|_{\beta}\le Ct^{-\frac{N}{\theta}\left(\frac{1}{\alpha}-\frac{1}{\beta}\right)}\left\|\phi\right\|_{\alpha}
\end{equation}
and \cite[Corollary~3.1]{MT06} with minor modifications.
See \cite{FKRT17,HO10} for (\ref{S2L3E2}).
The inequality (\ref{S2L3E1}) immediately follows from (\ref{S2L3E0}).
\end{proof}
We show that the constant $C_0>0$ in (\ref{S2L3E1}) can be chosen arbitrarily small if $\phi\in\calL^{\alpha}_{\rm{ul},\rho}(\RN)$.
This property is used in the critical case.
\begin{proposition}\label{S2S2P2}
Let $N\ge 1$ and $0<\theta\le 2$, and let $1\le\alpha<\beta\le\infty$.
For each $\phi\in\calL^{\alpha}_{\rm{ul},\rho}(\RN)$ and each $C_0>0$, there is $t_0=t_0(\phi,C_0)>0$ such that
\begin{equation}\label{S2S2P2E0}
\left\|S(t)\phi\right\|_{L^{\beta}_{\rm{ul},\rho}(\RN)}\le C_0t^{-\frac{N}{\theta}\left(\frac{1}{\alpha}-\frac{1}{\beta}\right)}\ \ \textrm{for}\ \ 0<t<t_0.
\end{equation}
\end{proposition}
\begin{proof}
We follow the proof of \cite[Lemma~8]{BC96}.
Let $\gamma:=\frac{N}{\theta}\left(\frac{1}{\alpha}-\frac{1}{\beta}\right)$.
By (\ref{S2L3E1}) we see the following:
For any $\psi\in L^{\infty}$, there is $t_0>0$ such that
\begin{align*}
t^{\gamma}\left\|S(t)\phi\right\|_{L^{\beta}_{\rm{ul},\rho}(\RN)}
&\le t^{\gamma}\left\|S(t)(\phi-\psi)\right\|_{L^{\beta}_{\rm{ul},\rho}(\RN)}+t^{\gamma}\left\|S(t)\psi\right\|_{L^{\beta}_{\rm{ul},\rho}(\RN)}\\
&\le C\left\|\phi-\psi\right\|_{L^{\alpha}_{\rm{ul},\rho}(\RN)}+Ct^{\gamma}\left\|\psi\right\|_{\infty}
\end{align*}
for $0<t<t_0$.
Then,
\[
\limsup_{t\to 0}t^{\gamma}\left\|S(t)\phi\right\|_{L^{\beta}_{\rm{ul},\rho}(\RN)}\le C\left\|\phi-\psi\right\|_{L^{\alpha}_{\rm{ul},\rho}(\RN)}.
\]
Since $\phi\in\calL^{\alpha}_{\rm{ul},\rho}(\RN)$, it follows from the definition of $\calL^{\alpha}_{\rm{ul},\rho}(\RN)$ that we can choose $\psi\in {BUC}(\RN)(\subset L^{\infty}(\RN))$ such that $\left\|\phi-\psi\right\|_{L^{\alpha}_{\rm{ul},\rho}(\RN)}$ is arbitrarily small.
Thus, (\ref{S2S2P2E0}) holds.
\end{proof}

\subsection{Preliminaries}
First we recall the monotonicity method.
\begin{lemma}\label{S2L-1}
Let $0<T\le\infty$ and let $f$ be a continuous nondecreasing function such that $f(0)\ge 0$.
The problem (\ref{S1E1}) has a solution in the sense of Definition~\ref{S1D1} if and only if (\ref{S1E1}) has a supersolution.
\end{lemma}
\begin{proof}
This lemma is well known.
See \cite[Theorem~2.1]{RS13} for details.
However, we briefly show the proof for readers' convenience.

If (\ref{S1E1}) has a solution, then the solution is also a supersolution.
Thus, it is enough to show that (\ref{S1E1}) has a solution if it has a supersolution.
Let $\bu$ be a supersolution in $\RN\times(0,T)$.
Let $u_1=S(t)\phi$.
We define $u_n$, $n=2,3,\ldots$, by
\[
u_n=\calF[u_{n-1}].
\]
Then we can show by induction that
\[
0\le u_1\le u_2\le\cdots\le u_n\le\cdots\le\bu <\infty\ \ \textrm{a.e.}\ x\in\RN,\ 0<t<T.
\]
This indicates that the limit $\lim_{n\to\infty}u_n(x,t)$ which is denoted by $u(x,t)$ exists for almost all $x\in\RN$ and $0<t<T$.
By the monotone convergence theorem we see that
\[
\lim_{n\to\infty}\calF[u_{n-1}]=\calF[u],
\]
and hence $u=\calF[u]$.
Then, $u$ is a solution of (\ref{S1E1}) in the sense of Definition~\ref{S1D1}.
\end{proof}

Hereafter in this section we prove several useful lemmas.
\begin{lemma}\label{S2L2}
Let $f$ be a function such that (F1) with $q\ge 1$ holds.
If there are $q_0\in[q,\infty)$ and $u_0>0$ such that $f'(u)F(u)\le q_0$ for $u\ge u_0$, then $f(u)\le f(u_0)F(u_0)^{q_0}/F(u)^{q_0}$ for $u\ge u_0$.
\end{lemma}
\begin{proof}
Integrating $f'(u)/f(u)\le q_0/(f(u)F(u))$ over $[u_0,u]$, we have that $\log \left(f(u)/f(u_0)\right)\le q_0\log\left(F(u_0)/F(u)\right)$, and hence we obtain the conclusion.
\end{proof}

Let $f_q(u)$ be defined by (\ref{fq}).
We define $F_q(u)$ by
\begin{equation}\label{Fqq}
F_q(u):=\int_u^{\infty}\frac{dt}{f_q(t)}=
\begin{cases}
\frac{1}{p-1}u^{-p+1},\ \frac{1}{p}+\frac{1}{q}=1, & \textrm{if}\ q>1,\\
e^{-u} & \textrm{if}\ q=1.
\end{cases}
\end{equation}
\begin{lemma}\label{S2L1}
Let $f$ be a function such that (F1) with $q\ge 1$ holds.
Assume that there are $\alpha\in[q,\infty)$ and $u_0>0$ such that $f'(u)F(u)\le\alpha$ for $u\ge u_0$.
Let $\Phi_{\alpha}(u)$ be defined by
\begin{equation}\label{S2L1E0}
\Phi_{\alpha}(u):=F^{-1}_{\alpha}\left(F(u)\right)=
\begin{cases}
(\alpha-1)^{\alpha-1}F(u)^{-(\alpha-1)} & \textrm{if}\ \alpha>1,\\
-\log F(u) & \textrm{if}\ \alpha=1.
\end{cases}
\end{equation}
The the following (i) (ii) and (iii) hold:\\
(i) $\Phi_{\alpha}\in C^2(0,\infty)$,\\
(ii) $\Phi_{\alpha}'(u)>0$ for $u>0$,\\
(iii) $\Phi_{\alpha}''(u)\ge 0$ for $u\ge u_0$, and hence $\Phi(u)$ is convex in $[u_0,\infty)$.
\end{lemma}
\begin{proof}
It is clear that $\Phi_{\alpha}\in C^2$, since $F(u)\in C^2$.
We have
\begin{equation}\label{S2L1E1}
\Phi'_{\alpha}(u)=
\begin{cases}
\frac{(\alpha-1)^{\alpha}}{f(u)F(u)^{\alpha}} & \textrm{if}\ \alpha>1,\\
\frac{1}{f(u)F(u)} & \textrm{if}\ \alpha=1,
\end{cases}
\quad
\Phi''_{\alpha}(u)=
\begin{cases}
(\alpha-1)^{\alpha}\frac{\alpha-f'(u)F(u)}{f(u)^2F(u)^{\alpha+1}} & \textrm{if}\ \alpha>1.\\
\frac{1-f'(u)F(u)}{f(u)^2F(u)^2} & \textrm{if}\ \alpha=1,
\end{cases}
\end{equation}
The assertion (ii) follows from (\ref{S2L1E1}).
If $u\ge u_0$, then $\alpha-f'(u)F(u)\ge 0$, and hence (iii) follows from (\ref{S2L1E1}).
\end{proof}

\begin{proposition}\label{P1}
Let $N\ge 1$ and $0<\theta\le 2$.
Let $\Psi(\,\cdot\,)\in C[0,\infty)$ be a convex function, and let $\psi(x)$ be a nonnegative function.
The the following (i) and (ii) hold:\\
(i) If $\psi\in L^{\gamma}_{\rm{ul},\rho}(\RN)$ for some $1\le\gamma\le\infty$, then
\begin{equation}\label{P1E0}
\Psi(S(t)[\psi](x))\le S(t)[\Psi(\psi)](x)
\ \ \textrm{for}\ \ x\in\RN,\ t>0.
\end{equation}
(ii) Assume in addition that $\Psi\ge 0$ and $\Psi\not\equiv 0$.
If $\Psi(\psi)\in L^{\gamma}_{\rm{ul},\rho}(\RN)$ for some $1\le\gamma\le\infty$, then (\ref{P1E0}) holds.
\end{proposition}
\begin{proof}
Since we could not find a proof of Jensen's inequality for $\psi\in L^{\gamma}_{\rm{ul},\rho}(\RN)$ in literature, we show the proof.\\
(i) Since $L^{\gamma}_{\rm{ul},\rho}(\RN)\subset L^1_{\rm{ul},\rho}(\RN)$, we see that $\psi\in L^1_{\rm{ul},\rho}(\RN)$.
Let $\chi_{B(n)}(x)$ be the indicator function supported on $\overline{B(n)}:=\{x\in\RN|\ |x|\le n\}$, and let $\psi_n(x):=\psi(x)\chi_{B(n)}(x)$.
We can easily see that $\psi_n\in L^1(\RN)$. Since
\[
G(x-\,\cdot\,,t)\ge 0,\ \int_{\RN}G(x-y,t)dy=1\ \ \textrm{and}\ \ \psi_n\in L^1(\RN),
\]
by the classical Jensen's inequality we have
\begin{equation}\label{P1E1}
\Psi(S(t)[\psi_n])(x)\le S(t)[\Psi(\psi_n)](x).
\end{equation}
Since $\psi_n\le\psi$, we see that
\begin{equation}\label{P1E2}
S(t)[\Psi(\psi_n)](x)\le S(t)[\Psi(\psi)](x).
\end{equation}
On the other hand, by the monotone convergence theorem we see that,
\[
\int_{\RN}G(x-y,t)\psi_n(y)dy\to\int_{\RN}G(x-y,t)\psi(y) dy\ \ \textrm{as}\ \ n\to\infty,
\]
and hence
\begin{equation}\label{P1E3}
S(t)[\psi_n](x)\to S(t)[\psi](x)\ \ \textrm{as}\ \ n\to\infty.
\end{equation}
Since $\Psi$ is continuous, by (\ref{P1E1}), (\ref{P1E2}) and (\ref{P1E3}) we have
\[
S(t)[\Psi(\psi)](x)\ge\Psi(S(t)[\psi_n])(x)\to \Psi(S(t)[\psi])(x)\ \ \textrm{as}\ \ n\to\infty.
\]
We obtain (\ref{P1E0}).\\
(ii) If we show that $\psi\in L^1_{\rm{ul},\rho}(\RN)$, then we can use (i), and the conclusion holds.
Hereafter, we show that $\psi\in L^1_{\rm{ul},\rho}(\RN)$.
Because of the assumption on $\Psi$, there are $a>0$ and $b\in\R$ such that $\Psi(u)\ge au-b$ for $u\ge 0$.
When $\gamma=1$ or $\gamma=\infty$, there is $C_0>0$ such that
\begin{equation}\label{P1E4}
a\int_{B_y(\rho)}\psi(x)dx
\le b|B_y(\rho)|+\int_{B_y(\rho)}\Psi(\psi(x))dx<C_0
\ \ \textrm{uniformly for}\ y\in\RN.
\end{equation}
We see that $\psi\in L^1_{\rm{ul},\rho}(\RN)$,  and hence the proof is complete.
When $1<\gamma<\infty$, we have
\[
\int_{B_y(\rho)}\Psi(\psi(x))dx\le\left\|\Psi(\psi)\right\|_{L^{\gamma}(B_y(\rho))}\left\|1\right\|_{L^{\gamma'}(B_y(\rho))},
\]
where $\gamma':=\gamma/(1-\gamma)$.
By the same inequality as (\ref{P1E4}) we see that $\psi\in L^1_{\rm{ul},\rho}(\RN)$.
The proof is complete.
\end{proof}

\section{Existence in the subcritical case}
\subsection{Algebraic growth case}
In this subsection we mainly prove Theorem~\ref{A}~(i-1).
Let $q>1$ and $r>\max\{N/\theta,q-1\}$ be given in Theorem~\ref{A}~(i-1).
Let
\begin{equation}\label{S3E-2}
0<\e<\min\left\{\frac{\theta r}{N}-1,r-q+1,2(q-1)\right\}
\ \ \textrm{and}\ \ \delta:=\frac{\e}{2}.
\end{equation}
Let $f$ be a function such that (\ref{F1}) with $q>1$ holds.
Then, there is $u_1>0$ such that
\begin{equation}\label{S3E-1}
f'(u)F(u)\le q+\delta\ \ \textrm{for}\ \ u\ge u_1.
\end{equation}
For simplicity we write $q_0:=q+\delta$ and $p_0:=q_0/(q_0-1)$.
In particular, $1<q<q_0$.
\begin{lemma}\label{S3L2}
The following (i) and (ii) hold:\\
(i) There is $C>0$ such that
\[
F(u)^{\frac{-1}{p_0-1}+\e}\le Cu
\ \ \textrm{for}\ \ u\ge 1.
\]
(ii) There is $u_2>0$ such that if $u\ge u_2$, then
\[
F\left(\frac{u}{\sqrt{1+\sigma}}\right)\le (1+\sigma)^{p_0-1}F(u)
\ \ \textrm{for}\ \ 0\le\sigma\le 1.
\]
\end{lemma}
\begin{proof}
(i) Let $\xi(u):=\log u-\log F(u)^{\frac{1}{p_0-1}-\e}$.
Then,
\[
\xi'(u)=\frac{f(u)F(u)-\left(\frac{1}{p_0-1}-\e\right)u}{uf(u)F(u)}.
\]
Let $\eta(u):=f(u)F(u)-\left(\frac{1}{p_0-1}-\e\right)u$.
Then,
\begin{align*}
\eta'(u)
&=f'(u)F(u)-1-\frac{1}{p_0-1}+\e\\
&\to q-1-(q_0-1)+\e\qquad (u\to\infty)\\
&=\e/2>0.
\end{align*}
Thus, $\xi(u)$ is increasing for large $u$.
Since $\xi(u)$ is continuous on $u\ge 1$,  there is $C>0$ such that $\xi(u)>\log C$ for $u\ge 1$.
The conclusion of (i) holds.\\
(ii) Since $p_0-1=1/(q_0-1)$, we define $\xi(\sigma):=\log F(u)+\frac{1}{q_0-1}\log (1+\sigma)-\log F(\frac{u}{\sqrt{1+\sigma}})$.
Then,
\[
\xi'(\sigma)=\frac{1}{(1+\sigma)(q_0-1)}\frac{f(\frac{u}{\sqrt{1+\sigma}})F(\frac{u}{\sqrt{1+\sigma}})-\frac{(q_0-1)u}{2\sqrt{1+\sigma}}}{f(\frac{u}{\sqrt{1+\sigma}})F(\frac{u}{\sqrt{1+\sigma}})}.
\]
Let $\eta(v):=f(v)F(v)-(q_0-1)v/2$.
Then
\[
\eta'(v)=f'(v)F(v)-1-\frac{q_0-1}{2}\to\frac{q-1-\delta}{2}(>0)\ \ \textrm{as}\ \ v\to\infty.
\]
Thus, $\eta(v)\to\infty$ as $v\to\infty$, and hence there is $u_2>0$ such that $\eta(v)\ge 0$ for $v\ge u_2/\sqrt{2}$.
If $u\ge u_2$, then $f(\frac{u}{\sqrt{1+\sigma}})F(\frac{u}{\sqrt{1+\sigma}})-\frac{(q-1)u}{2\sqrt{1+\sigma}}=\eta(\frac{u}{\sqrt{1+\sigma}})\ge 0$ for $0\le\sigma\le 1$, and hence $\xi'(\sigma)\ge 0$ for $0\le\sigma\le 1$.
Since $\xi(0)=0$, we see that if $u\ge u_2$, then $\xi(\sigma)\ge 0$ for $0\le\sigma\le 1$.
The conclusion of (ii) holds.
\end{proof}

Hereafter, we define $u_0$ by
\[
u_0:=\max\{u_1,u_2\},
\]
where $u_1$ is given in (\ref{S3E-1}) and $u_2$ is given in Lemma~\ref{S3L2}~(ii).
\begin{lemma}\label{S3L1}
Let $N\ge 1$ and $0<\theta\le 2$.
Let $\psi(x)\in L^{\gamma}_{\rm{ul},\rho}(\RN)$, $1\le\gamma\le\infty$, be a function such that $\psi(x)\ge u_0$.
Then the following holds:
\begin{equation}\label{S3L1E0}
S(t)[\psi](x)\le F^{-1}\left(F_{q_0}\left(S(t)\left[F^{-1}_{q_0}\left(F(\psi)\right)\right]\right)\right)(x)
\ \ \textrm{for}\ \ x\in\RN\ \textrm{and}\ t>0,
\end{equation}
where $F_{q_0}$ is defined by (\ref{Fqq}) with $q=q_0$.
\end{lemma}
\begin{proof}
Let $\Phi_{\alpha}$ be defined by (\ref{S2L1E0}) with $\alpha=q_0$.
Applying Lemma~\ref{S2L1}~(iii) with $\alpha$, we see that $\Phi_{\alpha}(u)$ is convex in $[u_0,\infty)$.
Since $\psi\ge u_0$, by Proposition~\ref{P1} we have that $\Phi_{\alpha}(S(t)\psi)\le S(t)\Phi_{\alpha}(\psi)$.
It follows from (\ref{S2L1E1}) that $\Phi'_{\alpha}>0$, and hence $\Phi_{\alpha}^{-1}$ is increasing and
\begin{equation}\label{S3L1E1}
S(t)\psi\le\Phi^{-1}_{\alpha}(S(t)\Phi_{\alpha}(\psi)).
\end{equation}
Since $\Phi_{\alpha}(u)=F^{-1}_{\alpha}\left(F(u)\right)$ and $\Phi^{-1}_{\alpha}(u)=F^{-1}\left(F_{\alpha}(u)\right)$, the inequality (\ref{S3L1E0}) follows from (\ref{S3L1E1}).
\end{proof}

Let us introduce the following function:
\begin{align}
\bu(t):
&=\left(F^{-1}\circ F_{q_0}\circ (1+\sigma)S(t)\circ F_{q_0}^{-1}\circ F\right)(\phi_0)\label{SS}\\
&=F^{-1}\left(\left((1+\sigma)S(t)\left[F(\phi_0)^{\frac{-1}{p_0-1}}\right]\right)^{-(p_0-1)}\right),\nonumber
\end{align}
where $\phi_0(x):=\max\{\phi(x),u_0\}$, $0<\sigma\le 1$ and we define $\left((1+\sigma)S(t)\right)[u]=(1+\sigma)(S(t)[u])$.
By (\ref{S2L1E1}) we see that $F^{-1}_{q_0}(F(u))$ is increasing in $u$.
We easily see that
\begin{equation}\label{SSE1}
\bu\ge u_0.
\end{equation}
Since $\left\|\phi_0\right\|_{L^r_{\rm{ul},\rho}(\RN)}\le \left\|\phi\right\|_{L^r_{\rm{ul},\rho}(\RN)}+\left\|u_0 \right\|_{L^r_{\rm{ul},\rho}(\RN)}<\infty$, we see that $\phi_0\in L^r_{\rm{ul},\rho}(\RN)$.
By (\ref{S3L1E0}) and (\ref{SS}) we have
\begin{equation}\label{SS2}
S(t)\phi_0\le F^{-1}\left((1+\sigma)^{p_0-1}F(\bu)\right).
\end{equation}

\begin{lemma}\label{S3L3}
Let $N\ge 1$ and $0<\theta\le 2$.
Assume that $r>N/\theta$, $r>q-1$, that $\phi\ge 0$ and that $f$ satisfies (\ref{F1}) with $q>1$.
If $F(\phi)^{-r}\in L^1_{\rm{ul},\rho}(\RN)$, then there is $T>0$ such that $\bu(t)$ defined by (\ref{SS}) is a supersolution of (\ref{S1E1}) for $0<t<T$.
\end{lemma}
\begin{proof}
We show that $\calF[\bu]\le\bu$.
We note that $q_0\ge q>1$.
Since $\phi(x)\le\phi_0(x)$, by (\ref{SS2}) and Lemma~\ref{S3L2}~(ii) we have
\begin{equation}\label{S3L3E1}
S(t)\phi\le S(t)\phi_0\le F^{-1}\left((1+\sigma)^{p_0-1}F(\bu)\right)\le\frac{\bu}{\sqrt{1+\sigma}}.
\end{equation}
Because of (\ref{SSE1}), by Lemma~\ref{S2L2} we have that
\begin{equation}\label{S3L3E2}
f(\bu)\le\frac{f(u_0)F(u_0)^{q_0}}{F(\bu)^{q_0}}.
\end{equation}
By (\ref{S3E-2}) we see that $(p_0-1)r=r/(q+\e-1)\ge 1$.
By (\ref{S2L3E1}) we see that there is $T>0$ such that
\begin{equation}\label{S32L3E2+}
\left\|S(t)F(\phi_0)^{\frac{-1}{p_0-1}}\right\|_{\infty}
\le C t^{\frac{-N}{\theta r (p_0-1)}}
\left\|F(\phi_0)^{\frac{-1}{p_0-1}}\right\|_{L^{(p_0-1)r}_{\rm{ul},\rho}(\RN)}
\ \ \textrm{for}\ \ 0<t<T.
\end{equation}
Using (\ref{S3L3E2}), (\ref{S32L3E2+}) and Lemma~\ref{S3L2}~(i), we have
\begin{align*}
\int_0^tS(t-s)&f(\bu(s))ds
\le C\int_0^tS(t-s)\left[F[\bu]^{-q_0}\right]ds\\
&=C(1+\sigma)^{p_0}\int_0^tS(t-s)\left[\left(S(s)\left[F(\phi_0)^{\frac{-1}{p_0-1}}\right]\right)
\left(S(s)\left[F(\phi_0)^{\frac{-1}{p_0-1}}\right]\right)^{p_0-1}\right]ds\\
&\le C(1+\sigma)^{p_0}S(t)\left[F(\phi_0)^{\frac{-1}{p_0-1}}\right]\int_0^t\left\|\left(S(s)\left[F(\phi_0)^{\frac{-1}{p_0-1}}\right]\right)^{p_0-1}\right\|_{\infty}ds\\
&= C(1+\sigma)^{p_0}\left(S(t)\left[F(\phi_0)^{\frac{-1}{p_0-1}}\right]\right)^{1-(p_0-1)\e}\left\|S(t)\left[F(\phi_0)^{\frac{-1}{p_0-1}}\right]\right\|^{(p_0-1)\e}_{\infty}\\
&\qquad\times \int_0^t\left\|S(s)\left[F(\phi_0)^{\frac{-1}{p_0-1}}\right]\right\|^{p_0-1}_{\infty}ds\\
&\le C(1+\sigma)^{(p_0-1)(1+\e)}F(\bu)^{\frac{-1}{p_0-1}+\e}\left(Ct^{-\frac{N}{\theta r(p_0-1)}}\left\|F(\phi_0)^{\frac{-1}{p_0-1}}\right\|_{L^{(p_0-1)r}_{\rm{ul},\rho}(\RN)}\right)^{(p_0-1)\e}\\
&\qquad\times \int_0^t \left(Cs^{-\frac{N}{\theta r(p_0-1)}}\left\|F(\phi_0)^{\frac{-1}{p_0-1}}\right\|_{L^{(p_0-1)r}_{\rm{ul},\rho}(\RN)}\right)^{(p_0-1)}ds\\
&\le C\bu(1+\sigma)^{(p_0-1)(1+\e)}\left\|F(\phi_0)^{\frac{-1}{p_0-1}}\right\|^{(1+\e)(p_0-1)}_{L^{(p_0-1)r}_{\rm{ul},\rho}(\RN)}
t^{-\frac{N\e}{\theta r}}\int_0^ts^{-\frac{N}{\theta r}}ds\\
&\le C\bu(1+\sigma)^{(p_0-1)(1+\e)}\left\|F(\phi_0)^{-r}\right\|_{L^1_{\rm{ul},\rho}(\RN)}^{\frac{1+\e}{r}}t^{1-\frac{(1+\e)N}{\theta r}}.
\end{align*}
Here, $\int_0^ts^{-\frac{N}{\theta r}}ds$ in the above calculation is integrable.
By (\ref{S3E-2}) we see that $1-(1+\e)N/(\theta r)>0$, and hence there is a small $T>0$ such that if $0<t<T$, then
\begin{equation}\label{S3L3E3}
C(1+\sigma)^{(p_0-1)(1+\e)}
\left\|F(\phi_0)^{-r}\right\|_{L^1_{\rm{ul},\rho}(\RN)}^{\frac{1+\e}{r}}
t^{1-\frac{(1+\e)N}{\theta r}}
<\frac{\sqrt{1+\sigma}-1}{\sqrt{1+\sigma}}.
\end{equation}
Here, we can choose $T>0$, which is still denoted by $T$, such that both (\ref{S32L3E2+}) and (\ref{S3L3E3}) hold.
Using (\ref{S3L3E3}) and (\ref{S3L3E1}), we have
\begin{align*}
\calF[\bu]
&=S(t)\phi+\int_0^tS(t-s)f(\bu(s))ds\\
&\le\frac{\bu}{\sqrt{1+\sigma}}+\frac{(\sqrt{1+\sigma}-1)\bu}{\sqrt{1+\sigma}}=\bu
\ \ \textrm{for}\ \ 0<t<T.
\end{align*}
Therefore, $\bu$ is a supersolution.
\end{proof}
\begin{lemma}\label{S3L4}
Let $N\ge 1$ and $0<\theta\le 2$.
Assume that $r>N/\theta$, $r\ge q-1$, that $\phi\ge 0$, that $f$ satisfies (\ref{F1}) with $q>1$ and that $f'(u)F(u)\le q$ for large $u>0$.
If $F(\phi)^{-r}\in L^1_{\rm{ul},\rho}(\RN)$, then there is $T>0$ such that $\bu(t)$ defined by (\ref{SS}) is a supersolution of (\ref{S1E1}) for $0<t<T$.
\end{lemma}
\begin{proof}
Let $\bu$ be given by (\ref{SS}) with $q_0=q$.
Because $f'(u)F(u)\le q$ for large $u>0$, we can show that $\bu$ is a supersolution as follows:
Lemma~\ref{S3L2}~(i) and (ii) hold if the proofs are slightly modified.
Lemmas~\ref{S3L1} holds without modification.
In the proof of Lemma~\ref{S3L3} we use $(p_0-1)r=r/(q-1)\ge 1$ instead of $(p_0-1)r=r/(q+\e-1)\ge 1$.
Then, the conclusion of Lemma~\ref{S3L4} holds.
The details are omitted.
\end{proof}

\subsection{Exponential growth case}
We consider the case $q=1$.
Let $r$ be given in Theorem~\ref{B}~(i).

\begin{lemma}\label{S32L1}
Assume that $f$ satisfies (\ref{F1}) with $q=1$ and (\ref{F2}).
For $\sigma>0$, $\alpha>0$ and $C_1>0$, there is $u_1>0$ such that
\[
F(u-C_1F(u)^{\alpha})\le e^{\sigma}F(u)\ \ \textrm{for}\ \ u\ge u_1.
\]
\end{lemma}
\begin{proof}
It is enough to show that
\begin{equation}\label{S32L1E1}
\lim_{u\to\infty}\frac{F(u-C_1F(u)^{\alpha})}{F(u)}=1.
\end{equation}
By L'Hospital's rule we have
\begin{equation}\label{S32L1E3}
\lim_{u\to\infty}\frac{F(u-C_1F(u)^{\alpha})}{F(u)}
=\lim_{u\to\infty}\frac{f(u)+C_1\alpha F(u)^{\alpha-1}}{f(u-C_1F(u)^{\alpha})}.
\end{equation}
Since $f$ is convex for large $u>0$, we have
\begin{equation}\label{S32L1E3+}
f(u-C_1F(u)^{\alpha})\ge f(u)-C_1f'(u)F(u)^{\alpha}\ \ \textrm{for large}\ u>0.
\end{equation}

First, we consider the case $\alpha\ge 1$.
We easily see that
\begin{equation}\label{S32L1E3++}
 f(u)\left(1-C_1f'(u)F(u)\frac{F(u)^{\alpha-1}}{f(u)}\right)
>0\ \ \textrm{for large $u>0$}.
\end{equation}
By (\ref{S32L1E3+}) and (\ref{S32L1E3++}) we  have
\[
1\le\lim_{u\to\infty}\frac{f(u)+C_1\alpha F(u)^{\alpha-1}}{f(u-C_1F(u)^{\alpha})}
\le\lim_{u\to\infty}\frac{f(u)+C_1\alpha F(u)^{\alpha-1}}{f(u)-C_1f'(u)F(u)^{\alpha}}
=\lim_{u\to\infty}\frac{1+C_1\alpha\frac{F(u)^{\alpha-1}}{f(u)}}{1-C_1f'(u)F(u)\frac{F(u)^{\alpha-1}}{f(u)}}
=1.
\]
Thus, the limit in (\ref{S32L1E3}) is $1$.
We obtain (\ref{S32L1E1}).

Second, we consider the case $0<\alpha<1$. 
We have
\[
(f(u)F(u)^{1-\alpha})'=\frac{f'(u)F(u)-1+\alpha}{F(u)^{\alpha}}\to\infty
\ \ \textrm{as}\ \ u\to\infty.
\]
Therefore, $f(u)F(u)^{1-\alpha}\to\infty$ as $u\to\infty$.
We easily see that
\begin{equation}\label{S32L1E4-}
f(u)\left(1-C_1\frac{f'(u)F(u)}{f(u)F(u)^{1-\alpha}}\right)
>0\ \ \textrm{for large $u>0$.}
\end{equation}
By (\ref{S32L1E3+}) and (\ref{S32L1E4-}) we have
\[
1\le\lim_{u\to\infty}\frac{f(u)+C_1\alpha F(u)^{\alpha-1}}{f(u-C_1F(u)^{\alpha})}
\le\lim_{u\to\infty}\frac{f(u)+C_1\alpha F(u)^{\alpha-1}}{f(u)-C_1f'(u)F(u)^{\alpha}}
=\lim_{u\to\infty}\frac{1+\frac{C_1\alpha}{f(u)F(u)^{1-\alpha}}}{1-C_1\frac{f'(u)F(u)}{f(u)F(u)^{1-\alpha}}}=1.
\]
We see that the limit in (\ref{S32L1E3}) is $1$. 
We obtain (\ref{S32L1E1}).
\end{proof}
Because of (\ref{F2}), there is $u_2>0$ such that $f(u)$ is convex on $[u_2,\infty)$ and that $f'(u)F(u)\le 1$ for $u>u_2$.

In this subsection we define $u_0$ by
\[
u_0:=\max\{u_1,u_2\},
\]
where $u_1$ is given Lemma~\ref{S32L1}.
\begin{corollary}\label{S32L2}
Let $N\ge 1$ and $0<\theta\le 2$.
Assume that $f$ satisfies (\ref{F1}) with $q=1$ and (\ref{F2}).
Let $\psi\in L^{\gamma}_{\rm{ul},\rho}(\RN)$, $1\le \gamma\le\infty$, be a function such that $\psi(x)\ge u_0$.
Then the following holds:
\[
S(t)[\psi](x)\le F^{-1}\left( F_1\left( S(t)\left[F^{-1}_1\left(F(\psi)\right)\right]\right)\right)(x)
\ \ \textrm{for}\ \ x\in\RN\ \textrm{and}\ t>0,
\]
where $F_1$ is defined by (\ref{Fqq}) with $q=1$.
\end{corollary}
\begin{proof}
The proof is the same as that of Lemma~\ref{S3L1} with $q_0=1$.
We omit the details.
\end{proof}

Let us introduce the following function:
\begin{align}
\bu(t)
&=\left(F^{-1}\circ F_1\circ \left(S(t)+\sigma\right)\circ F_1^{-1}\circ F\right)(\phi_0)\label{SS3}\\
&=F^{-1}\left(e^{-\sigma}\exp \left( S(t)\left[\log F(\phi_0)\right]\right)\right),\nonumber
\end{align}
where $\phi_0(x):=\max\{\phi(x),u_0\}$, $\sigma>0$ and we define $(S(t)+\sigma)[u]:=S(t)[u]+\sigma$.
We easily see that
\begin{equation}\label{S32E1}
\bu\ge u_0.
\end{equation}
Since $\left\|\phi_0\right\|_{L^r_{\rm{ul},\rho}(\RN)}\le \left\|\phi\right\|_{L^r_{\rm{ul},\rho}(\RN)}+\left\|u_0 \right\|_{L^r_{\rm{ul},\rho}(\RN)}<\infty$, we see that $\phi_0\in L^r_{\rm{ul},\rho}(\RN)$.
By Corollary~\ref{S32L2} and (\ref{SS3}) we have
\begin{equation}\label{SS4}
S(t)\phi_0\le F^{-1}\left(e^{\sigma}F(\bu)\right).
\end{equation}
\begin{lemma}\label{S32L3}
Let $N\ge 1$ and $0<\theta\le 2$.
Assume that $r>N/\theta$, that $\phi\ge 0$ and that $f$ satisfies (\ref{F1}) with $q=1$ and (\ref{F2}) hold.
If $F(\phi)^{-1}\in L^r_{\rm{ul},\rho}(\RN)$, then there is $T>0$ such that $\bu(t)$ defined by (\ref{SS3}) is a supersolution of (\ref{S1E1}) for $0<t<T$.
\end{lemma}
\begin{proof}
We show that $\calF[\bu]\le\bu$.
Because of (\ref{S32E1}), by Lemma~\ref{S2L2} we have
\begin{equation}\label{S32L3E1-}
f(\bu)\le\frac{f(u_0)F(u_0)}{F(\bu)}.
\end{equation}

First, we consider the case $r\ge 1$.
Using (\ref{S32L3E1-}), Proposition~\ref{P1} and (\ref{S2L3E1}), we have
\begin{align}
\int_0^tS(t-s)f(\bu(s))ds
&\le C\int_0^tS(t-s)\left[F(\bu)^{-1}\right]ds\nonumber\\
&= C\int_0^t e^{\sigma}S(t-s)\left[\exp\left(S(s)\left[\log F(\phi_0)^{-1}\right]\right)\right]ds\nonumber\\
&\le Ce^{\sigma}\int_0^t S(t-s)\left[S(s)\left[F(\phi_0)^{-1}\right]\right]ds\nonumber\\
&\le Ce^{\sigma}S(t)\left[F(\phi_0)^{-1}\right]\int_0^tds\nonumber\\
&\le CC_0e^{\sigma}\left\|F(\phi_0)^{-1}\right\|_{L^r_{\rm{ul},\rho(\RN)}}t^{1-\frac{N}{\theta r}}.\label{S32L3E1-+}
\end{align}
Using (\ref{SS4}) and (\ref{S32L3E1-+}) and Lemma~\ref{S32L1}, we have
\begin{align}
\calF[\bu]
&\le S(t)\phi_0+\int_0^tS(t-s)f(\bu(s))ds\nonumber\\
&\le F^{-1}\left(e^{\sigma}F(\bu)\right)+CC_0e^{\sigma}\left\|F(\phi_0)^{-r}\right\|^{\frac{1}{r}}_{L^1_{\rm{ul},\rho}(\RN)}t^{1-\frac{N}{\theta r}}\nonumber\\
&\le \bu-C_1F(\bu)^{\alpha}+CC_0e^{\sigma}\left\|F(\phi_0)^{-r}\right\|^{\frac{1}{r}}_{L^1_{\rm{ul},\rho}(\RN)}t^{1-\frac{N}{\theta r}},\label{S32L3E1}
\end{align}
where we define $\alpha:=\theta r/N-1>0$ and $C_1:=CC_0^{\theta r/N}e^{\sigma\theta r/N}\left\|F(\phi_0)^{-r}\right\|^{\theta/N}_{L^1_{\rm{ul},\rho}(\RN)}$.
By Proposition~\ref{P1} and (\ref{S2L3E1}) we have
\[
\frac{1}{F(\bu)}
=e^{\sigma}\exp\left(S(t)\left[\log F(\phi_0)^{-1}\right]\right)
\le e^{\sigma}S(t)\left[F(\phi_0)^{-1}\right]\le C_0e^{\sigma}\left\|F(\phi_0)^{-1}\right\|_{L^r_{\rm{ul},\rho}(\RN)}t^{-\frac{N}{\theta r}}.
\]
Hence, $F(\bu)^{\alpha}\ge C_0^{-\alpha}e^{-\alpha\sigma}\left\|F(\phi_0)^{-r}\right\|^{-\alpha/r}_{L^1_{\rm{ul},\rho}(\RN)}t^{\frac{\alpha N}{\theta r}}$.
By (\ref{S32L3E1}) we have
\begin{align*}
\calF[\bu]&\le\bu -C_1C_0^{-\alpha}e^{-\alpha\sigma}\left\|F(\phi_0)^{-r}\right\|^{-\frac{\alpha}{r}}_{L^1_{\rm{ul},\rho}(\RN)}t^{\frac{\alpha N}{\theta r}}
+CC_0e^{\sigma}\left\|F(\phi_0)^{-r}\right\|^{\frac{1}{r}}_{L^1_{\rm{ul},\rho}(\RN)}t^{1-\frac{N}{\theta r}}\\
&=\bu\ \ \textrm{for}\ \ 0<t<T,
\end{align*}
where we use $\frac{\alpha N}{\theta r}=1-\frac{N}{\theta r}>0$.
Thus, $\bu$ is a supersolution.

Second, we consider the case $r<1$.
By Proposition~\ref{P1} we have
\begin{equation}\label{S32L3E1+}
\frac{1}{F(\bu)}=e^{\sigma}\left(\exp\left(S(t)\left[\log F(\phi_0)^{-r}\right]\right)\right)^{\frac{1}{r}}
\le e^{\sigma}\left(S(t)\left[F(\phi_0)^{-r}\right]\right)^{\frac{1}{r}}.
\end{equation}
By (\ref{S32L3E1+}) and (\ref{S2L3E1}) we have
\begin{align}
\int_0^tS(t-s)f(\bu(s))ds
&\le C\int_0^tS(t-s)\left[F(\bu)^{-1}\right]ds\nonumber\\
&\le Ce^{\sigma}\int_0^tS(t-s)\left[S(s)\left[F(\phi_0)^{-r}\right]
\left\|S(s)\left[F(\phi_0)^{-r}\right]\right\|_{\infty}^{\frac{1}{r}-1}\right] ds\nonumber\\
&=Ce^{\sigma}S(t)\left[F(\phi_0)^{-r}\right]
\int_0^t\left\|S(s)\left[F(\phi_0)^{-r}\right]\right\|_{\infty}^{\frac{1}{r}-1}ds\nonumber\\
&\le C e^{\sigma}C_0t^{-\frac{N}{\theta}}\left\|F(\phi_0)^{-r}\right\|_{L^1_{\rm{ul},\rho}(\RN)}
\int_0^t\left(C_0s^{-\frac{N}{\theta}}\left\|F(\phi_0)^{-r}\right\|_{L^1_{\rm{ul},\rho(\RN)}}\right)^{\frac{1}{r}-1}ds\nonumber\\
&\le C C_0^{\frac{1}{r}}e^{\sigma}\left\|F(\phi_0)^{-r}\right\|_{L^1_{\rm{ul},\rho(\RN)}}^{\frac{1}{r}}
{t^{1-\frac{N}{\theta r}}}.\nonumber
\end{align}
Here, $\int_0^ts^{-\frac{N}{\theta}\left(\frac{1}{r}-1\right)}ds$ is integrable, since $-\frac{N}{\theta}\left(\frac{1}{r}-1\right)>-1$.
We define $\alpha:=\theta r/N-1$ and $C_1=CC_0^{\theta/N}e^{\sigma\theta r/N}\left\|F(\phi_0)^{-r}\right\|^{\theta/N}_{L^1_{\rm{ul},\rho}(\RN)}$.
Using (\ref{S32L3E1+}) and (\ref{S2L3E1}), we have
\[
\frac{1}{F(\bu)}
\le e^{\sigma}\left(S(t)\left[F(\phi_0)^{-r}\right]\right)^{\frac{1}{r}}
\le e^{\sigma}\left(C_0t^{-\frac{N}{\theta}}\left\|F(\phi_0)^{-r}\right\|_{L^1_{\rm{ul},\rho}(\RN)}\right)^{\frac{1}{r}}.
\]
We have
\[
\calF[\bu]\le\bu-C_1C_0^{-\frac{\alpha}{r}}e^{-\alpha\sigma}\left\|F(\phi_0)^{-r}\right\|^{-\frac{\alpha}{r}}_{L^1_{\rm{ul},\rho}(\RN)}t^{\frac{\alpha N}{\theta r}}
+CC_0^{\frac{1}{r}}e^{\sigma}\left\|F(\phi_0)^{-r}\right\|^{\frac{1}{r}}_{L^1_{\rm{ul},\rho}(\RN)}t^{1-\frac{N}{\theta r}}=\bu
\]
for $0<t<T$.
Thus, $\bu$ is a supersolution.
The proof is complete.
\end{proof}

\begin{proof}[Proof of Theorems~\ref{A}~(i-1), (i-2) and \ref{B}~(i)]
Theorem~\ref{A}~(i-1) (resp. (i-2)) follows from Lemmas~\ref{S2L-1} and \ref{S3L3} (resp. Lemmas~\ref{S2L-1} and \ref{S3L4}).
Theorem~\ref{B}~(i) follows from Lemmas~\ref{S2L-1} and \ref{S32L3}.
\end{proof}

\section{Nonexistence in the supercritical case}
We begin with a necessary condition for a local-in-time existence.
\begin{proposition}\label{S4P1}
Let $N\ge 1$ and $0<\theta\le 2$.
Assume that $f$ satisfies (\ref{F1}) with $q\ge 1$ and that $f(u)$ is convex for $u\ge 0$.
Let $\phi\in L^1_{\rm{ul},\rho}(\RN)$ be a nonnegative initial data.
If (\ref{S1E1}) has a nonnegative solution on $\RN\times (0,T)$ in the sense of Definition~\ref{S1D1}, then there is a small $T>0$ such that
\[
\left\|S(t)\phi \right\|_{\infty}\le F^{-1}(t)\ \ \textrm{for}\ \ 0<t<T.
\]
\end{proposition}
\begin{proof}
When $\theta=2$, the proof can be found in \cite[Lemma~4.1]{FI18}.
When $0<\theta<2$, the proof is also valid if the derivatives are understood in the weak sense.
We omit the proof.
\end{proof}

\begin{lemma}\label{S4L3}
Assume that $f$ satisfies (\ref{F1}) with $q>1$.
For $\beta>1$, there is $s_0>0$ such that
\[
F(s)^{\beta}\le F(\beta s)\ \ \textrm{for}\ \ s\ge s_0.
\]
\end{lemma}
\begin{proof}
Let $\xi(\gamma):=\log F(\gamma s)-\gamma\log F(s)$.
Then,
\[
\gamma\xi'(\gamma)-\xi(\gamma)=\frac{-\gamma s}{f(\gamma s)F(\gamma s)}-\log F(\gamma s).
\]
Let $\eta(\tau):=-\tau/(f(\tau)F(\tau))-\log F(\tau)$.
Since $\lim_{\tau\to\infty}(f(\tau)F(\tau))'=\lim_{\tau\to\infty}(f'(\tau)F(\tau)-1)=q-1>0$, we see that $\lim_{\tau\to\infty}f(\tau)F(\tau)=\infty$.
By L'Hospital's rule we have
\[
\lim_{\tau\to\infty}\left(\frac{-\tau}{f(\tau)F(\tau)}-\log F(\tau)\right)
=\lim_{\tau\to\infty}\left(\frac{-1}{f'(\tau)F(\tau)-1}-\log F(\tau)\right)=\infty,
\]
and hence there is $s_0>0$ such that $\eta(\tau)\ge 0$ for $\tau\ge s_0$.
If $\gamma s\ge s_0$, then  $\gamma\xi'(\gamma)-\xi(\gamma)\ge 0$.
Therefore, when $1\le\gamma\le\beta$ and $s\ge s_0$, we see that $\gamma s\ge s_0$, and hence $\gamma\xi'(\gamma)-\xi(\gamma)\ge 0$.
Since $(\xi(\gamma)/\gamma)'\ge 0$ and $\xi(1)=0$, we have that $\xi(\gamma)\ge 0$ for $1\le\gamma\le\beta$ and $s\ge s_0$.
The conclusion holds, since $\xi(\beta)\ge 0$.
\end{proof}

\begin{lemma}\label{S4L3+}
Assume that $f$ satisfies (\ref{F1}) with $q=1$ and that $f(u)$ is convex for large $u>0$.
For $\beta>1$, $\gamma>0$ and $C_1>0$, there is $s_1>0$ such that
\[
F(s)^{\beta}\le F(s+C_1F(s)^{\gamma})\ \ \textrm{for}\ \ s>s_1.
\]
\end{lemma}
\begin{proof}
It is enough to show that
\begin{equation}\label{S4L3+E1}
\lim_{s\to\infty}\frac{F(s)}{F(s+C_1F(s)^{\gamma})}=1,
\end{equation}
because
\[
\lim_{s\to\infty}\frac{F(s)^{\beta}}{F(s+C_1F(s)^{\gamma})}
=\lim_{s\to\infty}F(s)^{\beta-1}\frac{F(s)}{F(s+C_1F(s)^{\gamma})}=0.
\]
Since $F$ is convex, we see that
\begin{equation}\label{S4L3+E2}
F(s+C_1F(s)^{\gamma})\ge F(s)-C\frac{F(s)^{\gamma}}{f(s)}
\ \ \textrm{for large $s>0$}.
\end{equation}

First, we consider the case $\gamma\ge 1$.
Then $F(s)(1-C_1F(s)^{\gamma-1}/f(s))>0$ for large $s>0$.
Therefore, by (\ref{S4L3+E2}) we have
\[
1\le\lim_{s\to\infty}\frac{F(s)}{F(s+C_1F(s)^{\gamma})}
\le\lim_{s\to\infty}\frac{1}{1-C\frac{F(s)^{\gamma-1}}{f(s)}}=1.
\]
Thus, we obtain (\ref{S4L3+E1}).

Second, we consider the case $0<\gamma<1$. Since
\[
(f(s)F(s)^{1-\gamma})'=\frac{f'(s)F(s)-1+\gamma}{F(s)^{\gamma}}\to\infty\ \ \textrm{as}\ \ s\to\infty,
\]
we see that $f(s)F(s)^{1-\gamma}\to\infty$ as $s\to\infty$.
Then, $F(s)(1-C/(f(s)F(s)^{1-\gamma}))>0$ for large $s>0$.
Therefore, by (\ref{S4L3+E2}) we have
\[
1\le\lim_{s\to\infty}\frac{F(s)}{F(s+C_1F(s)^{\gamma})}
\le\lim_{s\to\infty}\frac{1}{1-\frac{C_1}{f(s)F(s)^{1-\gamma}}}=1.
\]
Thus, we obtain (\ref{S4L3+E1}).
The proof is complete.
\end{proof}

\begin{proof}[Proof of Theorems~\ref{A}~(ii) and \ref{B}~(ii)]
Let $r\in(0,N/\theta)$.
Then one can take $\alpha>0$ such that $\theta<\alpha<N/r$. Let 
\begin{equation}\label{S4PBE0}
u_0(x):=
\begin{cases}
F^{-1}(|x|^{\alpha}) & \textrm{if}\ F(0)=\infty,\\
F^{-1}\left(\min\{|x|^{\alpha},F(0)\}\right) & \textrm{if}\ F(0)<\infty.
\end{cases}
\end{equation}
Then, $F(u_0)^{-r}\in\unifcomp{1}$.
Let $\varepsilon>0$ so that $\alpha/\theta-\e\alpha>1$.

Suppose the contrary, i.e., (\ref{S1E1}) has a local-in-time nonnegative solution.
Let $K$ be given in Proposition~\ref{S2S2P1}.
Then, by Propositions~\ref{S4P1} and \ref{S2S2P1} we have
\begin{align}
	F^{-1}(t) &\geq \|S(t)u_0\|_{\infty}\nonumber\\
	&\geq t^{-N/\theta}\int_{\mathbb{R}^N}K(t^{-1/\theta}|y|)F^{-1}(|y|^{\alpha})dy\nonumber\\
	&=\int_{\RN}K(|z|)F^{-1}(t^{\alpha/\theta}|z|^{\alpha})dz\nonumber\\
	&\geq\int_{|z|\leq t^{-\e}}K(|z|)F^{-1}(t^{\alpha/\theta}|z|^{\alpha})dz\nonumber\\
	&\geq F^{-1}(t^{\beta})\int_{|z|\leq t^{-\e}}K(|z|)dz\nonumber\\
	&=F^{-1}(t^{\beta})\left(1-\int_{|z|> t^{-\e}}K(|z|)dz\right),\label{S4PBE0+}
\end{align}
where $\beta={\alpha}/{\theta}-\e\alpha>1$. Among other things, we used the fact that $F^{-1}$ is decreasing. Now, we have
\begin{align*}
	\int_{|z|> t^{-\e}}K(|z|)dz
	&= \omega_{N-1}\int_{t^{-\e}}^{\infty}\tau^{N-1}K(\tau)d\tau \\
	&=\omega_{N-1}\int_{t^{-\e}}^{\infty}\frac{1}{\tau^{\theta+1}}\left(\tau^{N+\theta}K(\tau)\right)d\tau,
\end{align*}
where $\omega_{N-1}$ denotes the area of the unit sphere in $\mathbb{R}^N$.
By Proposition~\ref{S2S2P1} 
we see that $\tau^{N+\theta}K(\tau)\leq C$ for $\tau\ge t^{-\e}$.
Hence, as $\theta>0$, we have
\begin{equation}\label{S4PBE0++}
\int_{|z|> t^{-\e}}K(|z|)dz\leq Ct^{\e\theta}.
\end{equation}
Note that (\ref{S4PBE0++}) also holds for $\theta=2$.
Therefore, if $0<\theta\le 2$, then by (\ref{S4PBE0+}) and (\ref{S4PBE0++}) we have
\begin{equation}\label{S4PBE1}
F^{-1}(t)\geq F^{-1}(t^{\beta})(1-Ct^{\varepsilon\theta}).
\end{equation}

First, we consider the case $q>1$.
Let $t=F(s)$.
By (\ref{S4PBE1}) and Lemma~\ref{S4L3}  we have
\begin{align*}
F^{-1}(F(s))
&\ge F^{-1}(F(s)^{\beta})(1-CF(s)^{\e\theta})\\
&\ge F^{-1}(F(\beta s))(1-CF(s)^{\e\theta}).
\end{align*}
Therefore, we have
\[
0\ge s(\beta-1-CF(s)^{\e\theta}).
\]
The above inequality does not hold for large $s>0$.
We obtain a contradiction, and hence the solution does not exist when $q>1$.
The proof of Theorem~\ref{A}~(ii) is complete.

Second, we consider the case $q=1$.
Let $t=F(s)$ and $\gamma=\e\theta/2$.
By (\ref{S4PBE1}) and Lemma~\ref{S4L3+} we have
\begin{align*}
F^{-1}(F(s))&\ge F^{-1}\left(F(s)^{\beta}\right)(1-CF(s)^{\e\theta})\\
&\ge F^{-1}(F(s+C_1F(s)^{\gamma}))(1-CF(s)^{\e\theta})\\
&=s+C_1F(s)^{\gamma}-CsF(s)^{\e\theta}-CC_1F(s)^{\e\theta+\gamma}.
\end{align*}
Then,
\begin{equation}\label{S4PBE2-}
0\ge F(s)^{\gamma}(C_1-CsF(s)^{\e\theta-\gamma}-CC_1F(s)^{\e\theta}).
\end{equation}
If we assume that
\begin{equation}\label{S4PBE2}
sF(s)^{\e\theta/2}\to 0\ \ \textrm{as}\ \ s\to\infty,
\end{equation}
then we have a contradiction, because the right-hand side of (\ref{S4PBE2-}) is positive for large $s>0$.

It is enough to prove (\ref{S4PBE2}).
Let $\delta:=\e\theta/4$.
Then, $f'(s)F(s)\le 1+\delta$ for large $s$, because of (\ref{F1}).
Integrating $f'(s)/f(s)\le (1+\delta)/(f(s)F(s))$ over $[s_0,s]$ twice, we have
\[
\frac{s-s_0}{f(s_0)F(s_0)^{1+\delta}}\le\frac{1}{\delta}\left(F(s)^{-\delta}-F(s_0)^{-\delta}\right).
\]
Thus,
\[
0\le sF(s)^{\frac{\e\theta}{2}}\le s\left(\frac{\delta(s-s_0)}{f(s_0)F(s_0)^{1+\delta}}+F(s_0)^{-\delta}\right)^{-\frac{\e\theta}{2\delta}}\to 0\ \ \textrm{as}\ \ s\to\infty,
\]
where we use $\e\theta/(2\delta)=2$.
The proof of Theorem~\ref{B}~(ii) is complete.
\end{proof}

\section{Existence in the critical case}
Theorems~\ref{A} and \ref{B} do not cover the critical case $F(\phi)^{-N/\theta}\in L^1_{\rm{ul},\rho}(\RN)$.
In this section we prove a local-in-time existence in the critical case when $f(u)=u^p$ or $e^u$.
We also give a simple generalization of the case $u^p$ or $e^u$.
\subsection{Pure power case}
\begin{theorem}\label{S5T1}
Let $N\ge 1$, $0<\theta\le 2$, $f(u)=u^p$, $p>1$, and $\phi\ge 0$.
Then the following hold:\\
(i) Assume that $r>N/\theta$ and $r\ge 1/(p-1)$.
If $\phi^{(p-1)r}\in L^1_{\rm{ul},\rho}(\RN)$, then (\ref{S1E1}) has a local-in-time solution in the sense of Definition~\ref{S1D1}.\\
(ii) Let $r:=N/\theta$ and $p>1+\theta/N$. If $\phi^{(p-1)r}\in\calL^1_{\rm{ul},\rho}(\RN)$, then (\ref{S1E1}) has a local-in-time solution in the sense of Definition~\ref{S1D1}.\\
(iii) For any $r\in(0,N/\theta)$, there is $\phi\ge 0$ such that $\phi^{(p-1)r}\in L^1_{\rm{ul},\rho}(\RN)$ has no local-in-time nonnegative solution in the sense of Definition~\ref{S1D1}.
\end{theorem}
Note that $r=N/\theta>q-1$, since $p>1+\theta/N$.
\begin{proof}
The assertions (i) and (iii) immediately follow from Theorem~\ref{A} (i-2) and (ii), respectively.

Hereafter, we prove (ii).
Because of Lemma~\ref{S2L-1}, it is enough to prove the existence of a supersolution.
The idea of the calculation (\ref{S5T1E2}) below comes from \cite[Section~4]{RS13}.
However, our supersolutions~(\ref{S5T1E0}) and (\ref{S5T1E2+-}) are simpler than $w(t)$ given in the proof of \cite[Theorem~4.4]{RS13}.
We divide the proof into two cases:
\begin{itemize}
\item[] Case (1): $1+\theta/N<p<N/(N-\theta)$,
\item[] Case (2): $p\ge N/(N-\theta)$.
\end{itemize}
Let $\alpha:=N(p-1)/\theta$.
Note that $p>\alpha$ if $p<N/(N-\theta)$, and $p\le\alpha$ if $p\ge N/(N-\theta)$.

{\it Case (1):} We consider the case where $1+\theta/N<p<N/(N-\theta)$.
Let $\sigma>0$ and
\begin{equation}\label{S5T1E0}
\bu(t):=(1+\sigma)(S(t)\phi^{\alpha})^{\frac{1}{\alpha}}.
\end{equation}
We show that $\bu$ is a supersolution.
By Proposition~\ref{P1} we have $(S(t)\phi)^{\alpha}\le S(t)\phi^{\alpha}$, and hence
\begin{equation}\label{S5T1E1}
S(t)\phi\le (S(t)\phi^{\alpha})^{\frac{1}{\alpha}}=\frac{\bu}{1+\sigma}.
\end{equation}
Since $\phi^{\alpha}\in \calL^1_{\rm{ul},\rho}(\RN)$, it follows from Proposition~\ref{S2S2P2} that there is $T>0$ such that
\begin{equation}\label{S5T1E1+}
\left\|S(t)\phi^{\alpha}\right\|_{\infty}\le C_0 t^{-N/\theta}\ \ \textrm{for}\ \ 0<t<T.
\end{equation}
Note that $p/\alpha>1$ and $\alpha>1$, since $1+\theta/N<p<N/(N-\theta)$.
Using (\ref{S5T1E0}) and (\ref{S5T1E1+}), we have
\begin{align}
\int_0^tS(t-s)\bu(s)^pds
&\le (1+\sigma)^p\int_0^tS(t-s)\left[\left(S(s)\phi^{\alpha}\right)^{\frac{p}{\alpha}}\right]ds\nonumber\\
&\le (1+\sigma)^p\int_0^tS(t-s)\left[S(s)\phi^{\alpha}\left\|S(s)\phi^{\alpha}\right\|_{\infty}^{\frac{p}{\alpha}-1}\right]ds\nonumber\\
&=(1+\sigma)^pS(t)\phi^{\alpha}\int_0^t\left\|S(s)\phi^{\alpha}\right\|_{\infty}^{\frac{p}{\alpha}-1}ds\nonumber\\
&=(1+\sigma)^p\left(S(t)\phi^{\alpha}\right)^{\frac{1}{\alpha}}\left\|S(t)\phi^{\alpha}\right\|_{\infty}^{1-\frac{1}{\alpha}}
\int_0^t\left\|S(s)\phi^{\alpha}\right\|_{\infty}^{\frac{p}{\alpha}-1}ds\nonumber\\
&=\bu (1+\sigma)^{p-1}\left\|S(t)\phi^{\alpha}\right\|_{\infty}^{1-\frac{1}{\alpha}}
\int_0^t\left\|S(s)\phi^{\alpha}\right\|_{\infty}^{\frac{p}{\alpha}-1}ds\nonumber\\
&\le\bu (1+\sigma)^{p-1}\left(C_0t^{-\frac{N}{\theta}}\right)^{1-\frac{1}{\alpha}}
\int_0^t\left(C_0s^{-\frac{N}{\theta}}\right)^{\frac{p}{\alpha}-1}ds\nonumber\\
&=\bu (1+\sigma)^{p-1}C_0^{\frac{p-1}{\alpha}}
t^{-\frac{N(\alpha-1)}{\theta\alpha}}\int_0^ts^{-\frac{N(p-\alpha)}{\theta \alpha}}ds\nonumber\\
&=\bu (1+\sigma)^{p-1}C_0^{\frac{p-1}{\alpha}}\frac{1}{1-\frac{N(p-\alpha)}{\theta\alpha}}.\label{S5T1E2}
\end{align}
Since $p>1+\theta/N$, we see that $-N(p-\alpha)/(\theta\alpha)>-1$, and hence $\int_0^ts^{-N(p-\alpha)/(\theta\alpha)}ds$ is integrable.
Because of Proposition~\ref{S2S2P2}, we can choose $C_0>0$ and $T>0$ such that
\begin{equation}\label{S5T1E2+}
(1+\sigma)^{p-1}C_0^{\frac{p-1}{\alpha}}\frac{1}{1-\frac{N(p-\alpha)}{\theta\alpha}}
\le\frac{\sigma}{1+\sigma}.
\end{equation}
Then, by (\ref{S5T1E1}), (\ref{S5T1E2}) and (\ref{S5T1E2+}) we have
\begin{align*}
\calF[\bu]&:=S(t)\phi+\int_0^tS(t-s)\bu(s)^pds\\
&\le\frac{\bu}{1+\sigma}+\frac{\sigma}{1+\sigma}\bu=\bu
\ \ \textrm{for}\ \ 0<t<T.
\end{align*}
Since $\calF[\bu]\le\bu$ for $0<t<T$, $\bu$ is a supersolution.
It follows from Lemma~\ref{S2L-1} that (\ref{S1E1}) has a local-in-time solution.

{\it Case (2):} We consider the case where $p\ge N/(N-2)$.
Let $\sigma>0$ and
\begin{equation}\label{S5T1E2+-}
\bu(t):=(1+\sigma)(S(t)\phi^p)^{\frac{1}{p}}.
\end{equation}
We show that $\bu$ is a supersolution.
Since $S(t)\phi\le(S(t)\phi^p)^{1/p}$, we have
\[
S(t)\phi\le \left(S(t)\phi^p\right)^{\frac{1}{p}}=\frac{\bu}{1+\sigma}.
\]
Since $\alpha/p\ge 1$ and $\phi^{\alpha}\in\calL^1_{\rm{ul},\rho}(\RN)$, by Proposition~\ref{S2S2P2} we have
\begin{equation}\label{S5T1E3}
\left\|S(t)\phi^p\right\|_{\infty}\le C_0t^{-\frac{Np}{\theta \alpha}}.
\end{equation}
Note that $p>1$.
By (\ref{S5T1E2+-}) and (\ref{S5T1E3}) we have
\begin{align*}
\int_0^tS(t-s)\bu(s)^pds
&\le (1+\sigma)^p\int_0^tS(t-s)\left[S(s)\phi^p\right]ds\\
&=(1+\sigma)^pS(t)\phi^p\int_0^tds\\
&\le(1+\sigma)^p\left(S(t)\phi^p\right)^{\frac{1}{p}}
\left\|S(t)\phi^p\right\|_{\infty}^{\frac{p-1}{p}}t\\
&\le\bu (1+\sigma)^{p-1}\left(C_0t^{-\frac{Np}{\theta \alpha}}\right)^{\frac{p-1}{p}}t\\
&=\bu (1+\sigma)^{p-1}C_0^{\frac{p-1}{p}}.
\end{align*}
Because of Proposition~\ref{S2S2P2}, we can choose $C_0>0$ and $T>0$ such that
\[
(1+\sigma)^{p-1}C_0^{\frac{p-1}{p}}\le\frac{\sigma}{1+\sigma}.
\]
The rest of the proof is the same as the case $1+\theta/N<p<N/(N-\theta)$.
We omit the details.
\end{proof}
\begin{remark}
In the Laplacian case $\theta=2$, the exponent $p=N/(N-2)$ is called \lq\lq doubly critical" in \cite[Remark 5]{BC96}.
This exponent is obtained by the relation $N(p-1)/2=p$.
It is known that the uniqueness of the solution to $\partial_tu=\Delta u+|u|^{2/(N-2)}u$ in $L^{N/(N-2)}(\RN)$ does not hold.
See \cite{NS85,T02}.
\end{remark}

\subsection{Pure exponential case}
\begin{theorem}\label{S5T2}
Let $N\ge 1$, $0<\theta\le 2$, $f(u)=e^u$ and $\phi\ge 0$.
Then the following hold:\\
(i) Assume that $r>N/\theta$. If $e^{r\phi}\in L^1_{\rm{ul},\rho}(\RN)$, then (\ref{S1E1}) has a local-in-time solution in the sense of Definition~\ref{S1D1}.\\
(ii) Let $r:=N/\theta$. If $e^{r\phi}\in\calL^1_{\rm{ul},\rho}(\RN)$, then (\ref{S1E1}) has a local-in-time solution in the sense of Definition~\ref{S1D1}.\\
(iii) For any $r\in (0,N/\theta)$, there is $\phi\ge 0$ such that $e^{r\phi}\in L^1_{\rm{ul},\rho}(\RN)$ and (\ref{S1E1}) has no local-in-time nonnegative solution in the sense of Definition~\ref{S1D1}.
\end{theorem}
\begin{proof}
The assertions (i) and (iii) immediately follow from Theorem~\ref{B}~(i) and (ii), respectively.

Hereafter, we prove (ii).
Because of Lemma~\ref{S2L-1}, it is enough to prove the existence of a supersolution.
Let $\sigma>0$ and
\begin{equation}\label{S5T2E0}
\bu(t):=S(t)\phi+\sigma.
\end{equation}
Note that (\ref{SS3}) becomes (\ref{S5T2E0}).
We show that (\ref{S5T2E0}) is a supersolution.
The proof is divided into two cases: $r\ge 1$ and $r<1$.

{\it Case (1):} We consider the case $r\ge 1$.
By (\ref{S5T2E0}) we see that
\begin{equation}\label{S5T2E1}
S(t)\phi=\bu-\sigma.
\end{equation}
Since $r\ge 1$, we can easily see that $e^{\phi}\in\calL^r_{\rm{ul},\rho}(\RN)$ if $e^{r\phi}\in\calL^1_{\rm{ul},\rho}(\RN)$.
Since $e^{\phi}\in \calL^r_{\rm{ul},\rho}(\RN)$, it follows from Proposition~\ref{S2S2P2} that there is $T>0$ such that
\begin{equation}\label{S5T2E1+}
\left\|S(t)e^{\phi}\right\|_{\infty}\le C_0t^{-1}\ \ \textrm{for}\ \ 0<t<T.
\end{equation}
Using Proposition~\ref{P1} and (\ref{S5T2E1+}), we have
\begin{align}
\int_0^tS(t-s)[e^{\bu(s)}]ds
&=e^{\sigma}\int_0^tS(t-s)\left[\exp\left(S(s){\phi}\right)\right]ds\nonumber\\
&\le e^{\sigma}\int_0^tS(t-s)\left[S(s)e^{\phi}\right]ds\nonumber\\
&=e^{\sigma}S(t)e^{\phi}\int_0^tds\nonumber\\
&\le e^{\sigma}\left\|S(t)e^{\phi}\right\|_{\infty}t\nonumber\\
&\le e^{\sigma}C_0t^{-1}t\nonumber\\
&= C_0e^{\sigma}.\label{S5T2E2}
\end{align}
Because of Proposition~\ref{S2S2P2}, we can choose $C_0>0$ and $T>0$ such that
\begin{equation}\label{S5T2E3}
C_0e^{\sigma}\le\sigma.
\end{equation}
Then, by (\ref{S5T2E1}), (\ref{S5T2E2}) and (\ref{S5T2E3}) we have
\begin{align*}
\calF[\bu]
&:=S(t)\phi+\int_0^tS(t-s)e^{\bu(s)}ds\\
&\le \bu-\sigma+\sigma=\bu
\ \ \textrm{for}\ \ 0<t<T.
\end{align*}
Since $\calF[\bu]\le\bu$ for $0<t<T$, $\bu$ is a supersolution.
It follows from Lemma~\ref{S2L1} that (\ref{S1E1}) has a local-in-time solution.

{\it Case (2):} We consider the case $r<1$.
Since $S(t)\phi=\bu-\sigma$, we obtain (\ref{S5T2E1}).
Since $e^{r\phi}\in\calL^1_{\rm{ul},\rho}(\RN)$, it follows from Proposition~\ref{S2S2P2} that there is $T>0$ such that
\begin{equation}\label{S5T2E3+}
\left\|S(t)e^{r\phi}\right\|_{\infty}\le C_0t^{-r}\ \ \textrm{for}\ \ 0<t<T.
\end{equation}
Using Proposition~\ref{P1} and (\ref{S5T2E3+}), we have
\begin{align*}
\int_0^tS(t-s)e^{\bu(s)}ds
&=e^{\sigma}\int_0^tS(t-s)\left[\left(\exp\left(S(s)[r\phi]\right)\right)^{\frac{1}{r}}\right]ds\\
&\le e^{\sigma}\int_0^tS(t-s)\left[\left(S(s)\left[e^{r\phi}\right]\right)^{\frac{1}{r}}\right]ds\\
&\le e^{\sigma}\int_0^tS(t-s)\left[S(s)e^{r\phi}\left\|S(s)e^{r\phi}\right\|_{\infty}^{\frac{1}{r}-1}\right]ds\\
&= e^{\sigma}S(t)e^{r\phi}\int_0^t\left\|S(s)e^{r\phi}\right\|_{\infty}^{\frac{1}{r}-1}ds\\
&\le e^{\sigma}\left\|S(t)e^{r\phi}\right\|_{\infty}\int_0^t\left\|S(s)e^{r\phi}\right\|_{\infty}^{\frac{1}{r}-1}ds\\ 
&\le e^{\sigma}C_0t^{-r}\int_0^t\left(C_0s^{-r}\right)^{\frac{1-r}{r}}ds\\
&\le e^{\sigma}C_0t^{-r}\int_0^tC_0^{\frac{1-r}{r}}s^{-1+r}ds\\
&=e^{\sigma}C_0^{\frac{1}{r}}\frac{1}{r}.
\end{align*}
By Proposition~\ref{S2S2P2} we can choose $C_0>0$ and $T>0$ such that
\[
e^{\sigma}C_0^{\frac{1}{r}}\frac{1}{r}\le\sigma.
\]
The rest of the proof is the same as the case $r\ge 1$.
We omit the details.
\end{proof}

\subsection{Other nonlinearities}
Modifying the proofs of Theorems~\ref{S5T1}~(i) and \ref{S5T2}~(ii), we can easily prove the following:
\begin{corollary}
Let $N\ge 1$, $0<\theta\le 2$ and $\phi\ge 0$.
Assume that $f$ satisfies (\ref{F1}) with $q\ge 1$.
Then the following hold:\\
(i) Assume that there are $p>1+\theta/N$ and $C>0$ such that $f(u)\le Cu^p$ for large $u>0$.
If $\phi^{(p-1)N/\theta}\in\calL^1_{\rm{ul},\rho}(\RN)$, then (\ref{S1E1}) has a local-in-time solution in the sense of Definition~\ref{S1D1}.\\
(ii) Assume that there is $C>0$ such that $f(u)\le Ce^u$ for large $u>0$.
If $e^{N\phi/\theta}\in\calL^1_{\rm{ul},\rho}(\RN)$, then (\ref{S1E1}) has a local-in-time solution in the sense of Definition~\ref{S1D1}.
\end{corollary}
The details of the proofs are omitted.

\section{Summary and conjectures}
We study integrability conditions for a local-in-time existence and nonexistence of positive solutions of (\ref{S1E1}) when the initial data is positive and in uniformly local $L^p$ spaces.
The exponent $N(p-1)/\theta$ becomes a threshold, and we construct a local-in-time positive solution in the subcritical case (Theorems~\ref{A}~(i-1), (i-2) and \ref{B}~(i)), and show that there is an initial data such that (\ref{S1E1}) has no solution in the supercritical case (Theorems~\ref{A}~(ii) and \ref{B}~(ii)).
For $f(u)=u^p$ (resp. $e^u$), a local-in-time solution can be constructed in the critical case when $\phi^{(p-1)N/\theta}\in\calL^1_{\rm{ul},\rho}(\RN)$ (resp. $e^{r\phi}\in\calL^1_{\rm{ul},\rho}(\RN)$).
The following conjectures are left open:
\begin{conjecture}[{Existence for general $f$, critical case}]
Assume that $f$ satisfies (\ref{F1}) with $q\ge 1$, that $f'(u)F(u)\le q$ for large $u>0$ and $\phi\ge 0$.
Let $r=N/2>q-1$. If $F(\phi)^{-r}\in\calL^1_{\rm{ul},\rho}(\RN)$, then (\ref{S1E1}) has a local-in-time solution.
\end{conjecture}
\begin{conjecture}
Theorem~\ref{B}~(i) holds without the assumption $f'(u)F(u)\le 1$ for large $u>0$.
\end{conjecture}

\bigskip

\noindent
{\bf Acknowledgements}

The second author was supported by JSPS KAKENHI Grant Number 19H01797.


\begin{thebibliography}{99}
\bibitem{AD91}{D.~Andreucci and E.~DiBenedetto},
{\it On the Cauchy problem and initial traces for a class of evolution equations with strongly nonlinear sources},
Ann. Scuola Norm. Sup. Pisa Cl. Sci. {\bf 18} (1991) 363--441.
\bibitem{BSV17}{M.~Bonforte, Y.~Sire and J.~V\'{a}zquez},
{\it Optimal existence and uniqueness theory for the fractional heat equation},
Nonlinear Anal. {\bf 153} (2017), 142--168. 

\bibitem{BC96}{H.~Brezis and T.~Cazenave},
{\it A nonlinear heat equation with singular initial data},
J. Anal. Math. {\bf 68} (1996), 277--304.

\bibitem{DF10}{L.~Dupaigne and A.~Farina},
{\it Stable solutions of $-\Delta u=f(u)$ in $\RN$},
J. Eur. Math. Soc. {\bf 12} (2010), 855--882.

\bibitem{FI18}{Y.~Fujishima and N.~Ioku},
{\it Existence and nonexistence of solutions for the heat equation with a superlinear source term},
J. Math. Pures Appl. {\bf 118} (2018), 128--158. 

\bibitem{FKRT17}{G.~Furioli, T.~Kawakami, B.~Ruf and E.~Terraneo},
{\it Asymptotic behavior and decay estimates of the solutions for a nonlinear parabolic equation with exponential nonlinearity},
J. Differential Equations {\bf 262} (2017), 145--180.

\bibitem{G86}{Y.~Giga},
{\it Solutions for semilinear parabolic equations in $L^p$ and regularity of weak solutions of the Navier-Stokes system},
J. Differential Equations {\bf 62} (1986) 415--421.

\bibitem{HO10}{H.~Hayashi and T.~Ogawa},
{\it $L^p$-$L^q$ type estimate for the fractional order Laplacian in the Hardy space and global existence of the dissipative quasi-geostrophic equation},
Adv. Differ. Equ. Control Process. {\bf 5} (2010), 1--36. 

\bibitem{HI18}{K.~Hisa and K.~Ishige},
{\it Existence of solutions for a fractional semilinear parabolic equation with singular initial data},
Nonlinear Anal. {\bf 175} (2018), 108--132.

\bibitem{LRSV16}{R.~Laister, J.~Robinson, M.~Sier\.{z}\c{e}ga and A.~Vidal-L\'{o}pez},
{\it A complete characterisation of local existence for semilinear heat equations in Lebesgue spaces},
Ann. Inst. H. Poincar\'e Anal. Non Lin\'{e}aire {\bf 33} (2016), 1519--1538.

\bibitem{L17}{K.~Li},
{\it A characteristic of local existence for nonlinear fractional heat equations in Lebesgue spaces},
Comput. Math. Appl. {\bf 73} (2017), 653--665. 

\bibitem{L17b}{K.~Li},
{\it No local $L^1$ solutions for semilinear fractional heat equations},
Fract. Calc. Appl. Anal. {\bf 20} (2017), 1328--1337.

\bibitem{MT06}{Y.~Maekawa and Y.~Terasawa},
{\it The Navier-Stokes equations with initial data in uniformly local $L^p$ spaces},
Differential Integral Equations {\bf 19} (2006), 369--400.

\bibitem{M18}{Y.~Miyamoto},
{\it A limit equation and bifurcation diagrams of semilinear elliptic equations with general supercritical growth},
J. Differential Equations {\bf 264} (2018), 2684--2707.

\bibitem{NS85}{W.~Ni and P.~Sacks},
{\it Singular behavior in nonlinear parabolic equations},
Trans. Amer. Math. Soc. {\bf 287} (1985), 657--671.

\bibitem{QS07}{P.~Quittner and P.~Souplet},
{Superlinear parabolic problems. Blow-up, global existence and steady states},
{\it Birkh\"{a}user Advanced Texts: Basler Lehrb\"{u}cher. Birkh\"{a}user Verlag, Basel, 2007. xii+584 pp. ISBN: 978-3-7643-8441-8.}

\bibitem{RS13}{J.~Robinson and M.~Sier\.{z}\c{e}ga},
{\it Supersolutions for a class of semilinear heat equations},
Rev. Mat. Complut. {\bf 26} (2013), 341--360. 

\bibitem{S75}{S. Sugitani},
{\it On nonexistence of global solutions for some nonlinear integral equations},
Osaka J. Math. {\bf 12} (1975) 45--51.

\bibitem{T02}{E.~Terraneo},
{\it Non-uniqueness for a critical non-linear heat equation},
Comm. Partial Differential Equations {\bf 27} (2002), 185--218.

\bibitem{W80}{F.~Weissler},
{\it Local existence and nonexistence for semilinear parabolic equations in $L^p$},
Indiana Univ. Math. J. {\bf 29} (1980), 79--102. 

\bibitem{W86}{F.~Weissler},
{\it $L^p$-energy and blow-up for a semilinear heat equation},
Nonlinear functional analysis and its applications, Part 2 (Berkeley, Calif., 1983), 545--551,
Proc. Sympos. Pure Math., 45, Part 2, Amer. Math. Soc., Providence, RI, 1986.

\end{thebibliography}
\end{document}